\documentclass[a4paper,11pt]{article}
\usepackage{amsmath}                    
\usepackage{amssymb}                    
\usepackage{amsthm}                     
\usepackage{amsfonts}                   
\usepackage{subcaption}
\usepackage[usenames, dvipsnames]{color}
\usepackage{graphics,graphicx}
\usepackage{enumerate}

\usepackage{dsfont}
\usepackage{float}

\allowdisplaybreaks

\newtheorem{theorem}{Theorem}[section]

\newtheorem{lemma}[theorem]{Lemma}
\newtheorem{remark}[theorem]{Remark}

\def\ts{\thinspace}
\DeclareMathOperator{\arcsinh}{arcsinh}

\DeclareMathOperator{\sinc}{sinc}
\renewcommand{\epsilon}{\varepsilon}

\newcommand{\R}{\mathbb{R}}
\newcommand{\bigo}{\mathcal{O}}

\newcommand{\IP}{\mathbb{P}}
\def \IR{\mathbb R}
\newcommand{\EE}{\IR^d}

\def \IC{\mathbb C}

\def \IE{\mathbb E}

\newcommand{\calN}{\mathcal{N}}
\newcommand{\calL}{\mathcal{L}}
\newcommand{\calS}{\mathcal{S}}

\renewcommand{\R}{\mathbb{R}}
\newcommand{\dif}[2]{\frac{\partial #1}{\partial #2}}
\newcommand{\diff}[2]{\frac{\partial^2 #1}{\partial #2^2}}

\title{
	Optimal explicit stabilized integrator of weak order one for\\
	stiff and ergodic stochastic differential equations
}

\author{ 
	Assyr Abdulle\textsuperscript{1}, Ibrahim Almuslimani\textsuperscript{2}, and Gilles Vilmart\textsuperscript{2}
}

\begin{document}
	\maketitle
	\footnotetext[1]{Mathematics Section, \'Ecole Polytechnique F\'ed\'erale de Lausanne,
		Station 8, CH-1015 Lausanne, Switzerland, Assyr.Abdulle@epfl.ch}
	\footnotetext[2]{Universit\'e de Gen\`eve, Section de math\'ematiques, 2-4 rue du Li\`evre, CP 64, CH-1211 Gen\`eve 4, Switzerland, Ibrahim.Almuslimani@unige.ch, Gilles.Vilmart@unige.ch}
	\begin{abstract}
A new explicit stabilized scheme of weak order one for stiff  and ergodic stochastic differential equations (SDEs)  is introduced. In the absence of noise, the new method coincides with the classical deterministic stabilized scheme (or Chebyshev method) for diffusion dominated advection-diffusion problems and it inherits its optimal stability domain size that grows quadratically with the number of internal stages of the method.
For mean-square stable stiff stochastic problems, the scheme has an  optimal extended mean-square stability domain that grows at the same quadratic rate as the deterministic stability domain size  in contrast to known existing methods for stiff SDEs [A. Abdulle and T. Li. Commun. Math. Sci., 6(4), 2008, A. Abdulle, G. Vilmart, and K. C. Zygalakis, SIAM J. Sci. Comput., 35(4), 2013].
Combined with postprocessing techniques, the new methods
achieve a convergence rate of order two for sampling the invariant measure of a class of ergodic SDEs, achieving a stabilized version of the non-Markovian scheme introduced in [B. Leimkuhler, C. Matthews, and M. V. Tretyakov, Proc. R. Soc. A, 470, 2014].

		\smallskip
		
	\end{abstract}
	
	
	
	\section{Introduction}
	
	We consider It\^o systems of stochastic differential equations of the form
	\begin{equation} \label{eq:sde0}
	dX(t) = f(X(t)) dt + \sum_{r=1}^mg^r(X(t)) dW_r(t),\qquad X(0)=X_0
	\end{equation}
	where $X(t)$ is a stochastic process with values in $\IR^d$, $f:\IR^d\rightarrow \IR^d$ is the drift term,
	$g^r:\IR^d\rightarrow \IR^d$, $r=1,\ldots, m$ are the diffusion terms, and $W_r(t)$, $r=1,\ldots,m,$ are independent 
	one-dimensional Weiner processes fulfilling the usual assumptions.
	We assume that the drift and diffusion functions are smooth enough and Lipschitz continuous to ensure the existence and uniqueness of a solution of \eqref{eq:sde0}
	on a given time interval $(0,T)$.
	We consider autonomous problems to simplify the presentation, but we emphasise that the scheme can also be extended to non-autonomous SDEs. A one step numerical integrator 
for the approximation of \eqref{eq:sde0} at time $t=nh$ is a discrete dynamical system of the form
\begin{equation}\label{eq:methnum}
X_{n+1} = \Psi(X_{n},h,\xi_{n})
\end{equation}
where $h$ denotes the stepsize and $\xi_{n}$ are independent random vectors. 
Analogously to the deterministic case, standard explicit numerical schemes for stiff stochastic problems, such as the simplest Euler-Maruyama method defined as
	\begin{equation} \label{eq:EM}
	X_{n+1} = X_n + hf(X_n) + \sum_{r=1}^mg^r(X_n) \Delta W_{n,r},\qquad X(0)=X_0,
	\end{equation}
	where $\Delta W_{n,r}=W_{r}(t_{n+1})-W_{r}(t_{n})$ are the Brownian increments,
face a severe timestep restriction \cite{Hig00,HaW96,KlP92}, and one can use an implicit or semi-implicit scheme with favourable
	stability properties. In particular, it is shown in \cite{Hig00} that the implicit $\theta$-method of weak order one is mean-square A-stable if and only if $\theta \geq 1/2$, while
	weak order two mean-square A-stable are constructed in \cite{AVZ12b}.
	An alternative approach is to consider explicit stabilized schemes with extended stability domains, as proposed in 
	\cite{AbC08,AbL08}.
	In \cite{AbL08} the deterministic Chebyshev method is extended to the context of mean-square stiff stochastic differential equations with It\^o noise, while the Stratonovitch noise case
	is treated in \cite{AbC08}.
	In place of a standard small damping, the main idea in \cite{AbC08,AbL08} is to use a large damping parameter $\eta$ optimized for each number $s$ of stages to stabilize
	the noise term. This yields a family of Runge-Kutta type schemes with extended stability domain with size $L_s\simeq 0.33 s^2$.
	This stability domain size was improved to $L_s\simeq 0.42 s^2$ in \cite{AVZ13} where a family of  weak second order stabilized schemes (and strong order one under suitable assumptions) is constructed based on the deterministic orthogonal Runge-Kutta-Chebyshev method of order 2 (ROCK2) \cite{AbM01}.

For ergodic SDEs, 	i.e., when \eqref{eq:sde0} has a unique invariant measure 
$\mu$ satisfying 
for each test function $\phi$ and for any deterministic initial condition $X_0=x$,
\begin{equation} \label{eq:ergodic_ex}
\lim_{T \rightarrow \infty} \frac{1}{T}\int_{0}^{T}\phi(X(s))ds =\int_{\EE}\phi(y)d\mu(y), \qquad \mbox{almost surely},
\end{equation}
one is interested in approximating numerically the long-time dynamics and to design numerical scheme with a unique invariant measure such that 
\begin{equation} \label{eq:difference1i}
\left| \lim_{N \rightarrow \infty} \frac{1}{N+1}\sum_{n=0}^{N} \phi(X_{n})-\int_{\EE}\phi(y)d\mu(y)\right|\leq C h^{r},
\end{equation} 
where $C$ is independent of $h$ small enough and $X_0$. In such a situation, we say that the numerical scheme has order $r$ with respect to the invariant measure.
For instance, the Euler-Maruyama method has order~$1$ with respect to the invariant measure. In \cite{LM13} the following non-Markovian scheme with the same cost as the Euler-Maruyama method was proposed for Brownian dynamics, i.e where the vector field is a gradient $f(x)=-\nabla V(x)$ and the noise is additive ($g(x)=\sigma$), 
	\begin{equation} \label{eq:LM}
	X_{n+1} = X_n + hf(X_n) + \sigma \frac{\Delta W_{n,j}+\Delta W_{n+1,j}}2,\qquad X(0)=X_0,
	\end{equation}
	and it was shown in \cite{LMT14}
that \eqref{eq:LM} has order $2$ with respect to the invariant measure for Brownian dynamics.
However, the admissible stepsizes for such an explicit method to be stable may face a severe restriction and alternatively to switching to drift-implicit methods, one may ask if a stabilized version of such an attractive non-Markovian scheme exists.

In this paper we introduce a new family of explicit stabilized schemes with optimal mean-square stability domain of size $L_s=Cs^2$, where $C\geq 2-\frac43 \eta$ and $\eta\geq 0$ is a small  parameter. We emphasize that in the deterministic case, $L_s=2s^2$ is the largest, i.e. optimal, stability domain along the negative real axis for an explicit s-stage Runge-Kutta method \cite{HaW96}. 
We note that the Chebyshev method \eqref{eq:Cheb1} (with $\eta=0)$ realizes such an optimal stability domain. The new schemes have strong order $1/2$ and weak order $1$.
The main ingredient for the design of the new schemes is to consider second kind Chebyshev polynomials, in addition to the usual first kind Chebyshev polynomials involved in the deterministic Chebychev method and stochastic extensions \cite{AbL08,AbC08}.
For stiff  stochastic problems, the stability domain sizes are close to the optimal value $2s^2$
and in the deterministic setting the method coincide with the optimal first order  explicit stabilized method. Thus these methods are more efficient than previously introduced stochastic stabilized methods \cite{AbL08,AVZ13}.
For ergodic dynamical systems, in the context of the ergodic Brownian dynamics, the new family of explicit stabilized schemes allows for a postprocessing \cite{Vil15} (see also \cite{BrV16,LaV17} in the context of Runge-Kutta methods) to achieve
order two of accuracy for sampling the invariant measure. In this context, our new methods can be seen as a stabilized version of the non-Markovian scheme \eqref{eq:LM} introduced in \cite{LM13,LMT14}.

	This paper is organized as follows.
	In Section \ref{sec:method}, we introduce the new family of schemes with optimal stability domain and we recall the main tools for the study of stiff integrators in the mean-square sense.
	We then analyze its mean-square stability properties (Section \ref{sec:analysis}), and convergence properties (Section \ref{sec:convergence}).
	In Section \ref{sec:langevin}, using a postprocessor we present a modification  with negligible overcost that yields order two of accuracy for the invariant measure of a class of ergodic overdamped Langevin equation.
	Finally, Section \ref{sec:numerical} is dedicated to the numerical experiments that confirm our theoretical analysis and illustrate the efficiency of the new schemes.

	\section{New second kind Chebyshev methods}
	\label{sec:method}
In this section we introduce our new stabilized stochastic method. We first briefly recall the concept of stabilized methods.
In the context of ordinary differential equations {(ODEs)},
\begin{equation} \label{eq:ode}
\frac{dX(t)}{dt} = f(X(t)), \qquad X(0)=X_0,
\end{equation}
and the Euler method $X_{1}=X_{0}+hf(X(0))$, a stabilization procedure  based on recurrence formula has been introduced
in \cite{HoS80}. Its construction relies on Chebyshev polynomials (hence the alternative name ``Chebyshev methods"),
$T_s(\cos x) =\cos (s x)$  and it is based on the explicit $s$-stage Runge-Kutta method
\begin{eqnarray} 
\label{eq:Cheb1}
K_{0} &=& X_{0}, \quad
K_{1} = X_{0}+h \mu_1 f(K_{0}), \nonumber \\
K_{i} &=&\mu_ihf(K_{i-1})+\nu_iK_{i-1}+\kappa_iK_{i-2},\quad j=2,\ldots,s,\\
X_{1} &=& K_{s} \nonumber,
\end{eqnarray}
where 
\begin{equation}\label{eq:omega}
\omega_{0}=1+\frac{\eta}{s^{2}},\quad \omega_{1}=\frac{T_{s}(\omega_{0})}{T'_{s}(\omega_{0})},\quad \mu_1=\frac{\omega_1}{\omega_0},
\end{equation} 
and for all $i=2,\ldots,s,$
	\begin{equation}\label{coeff}
	\mu_i=\frac{2\omega_1T_{i-1}(\omega_0)}{T_i(\omega_0)},\quad 
	\nu_i=\frac{2\omega_0T_{i-1}(\omega_0)}{T_i(\omega_0)},\quad 
	\kappa_i=-\frac{T_{i-2}(\omega_0)}{T_i(\omega_0)}=1-\nu_i.
	\end{equation}
One can easily check that the (family of) methods \eqref{eq:Cheb1} has the same first order accuracy as the Euler method (recovered for $s=1$). In addition, the scheme
\eqref{eq:Cheb1} has a low memory requirement (only two stages  should be stored when applying the recurrence formula) and it has a good internal stability with respect to round-off errors \cite{HoS80}.
The attractive feature of such a scheme comes from its stability behavior. Indeed, the method \eqref{eq:Cheb1} applied to the 
linear test problem $dX(t)/dt=\lambda X(t)$ yields, using the recurrence relation 
		\begin{equation} \label{eq:recTU1}
		T_j(p) = 2p T_{j-1}(p) - T_{j-2}(p),\qquad
		\end{equation}
		where $T_0(p)= 1,T_1(p)= p$, 
with $p=\lambda h$,
\begin{equation} \label{eq:Cheb1_stab}
X_{1}=R_{s,\eta}(p) X_0=\frac{T_{s}(\omega_0+\omega_1 p)}{T_{s}(\omega_0)} X_0,
\end{equation}
where the dependence of the stability function $R_{s,\eta}$ on the parameters $s$ and $\eta$ is emphasized with a corresponding subscript. 
The real negative interval $(-C_s(\eta)\cdot s^2 ,0)$ is included in the 
stability domain of the method
\begin{equation} 
\label{eq:stabSre}
{\cal S}:=\{p\in \mathbb{C}; |R_{s,\eta}(p)|\leq 1\}.
 \end{equation}
 The constant $C_s(\eta) \simeq 2-4/3\,\eta$ depends on the so-called damping parameter $\eta$ and for $\eta=0$,
 it reaches the maximal value $C_s(0)=2$. Hence, given the stepsize $h$, for systems with a Jacobian having large real negative eigenvalues (such as diffusion problems) with spectral radius $\lambda_{\max}$ at $X_n$, the parameter $s$ for the next step $X_{n+1}$ can be chosen adaptively as\footnote{The notation $[x]$ stands for the integer rounding of real numbers.}
\begin{equation} \label{eq:defsnum}
s=\left[\sqrt{\frac{h\lambda_{\max}+1.5}{2-4/3\,\eta}}+0.5\right],
\end{equation}
see \cite{Abd02} in the context of deterministically stabilized schemes of order two with adaptative stepsizes.
The method \eqref{eq:Cheb1}
 is much more efficient as its stability domain increases {\it quadratically} with the number $s$ of function evaluations while a composition of $s$  explicit Euler steps
 (same cost) has a stability domain that only increases {\it linearly} with $s$.
 \begin{figure}[tb]
		\smallskip
		\centering
		\begin{subfigure}[t]{0.9\textwidth}
			\centering
			\includegraphics[width=0.9\linewidth]{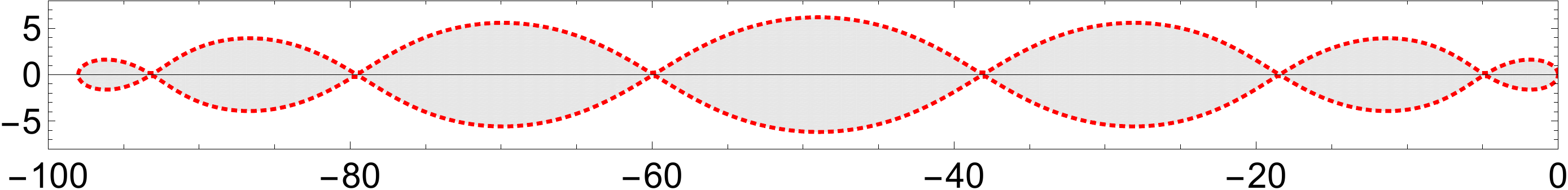}\\[2.ex]
			\includegraphics[width=0.9\linewidth]{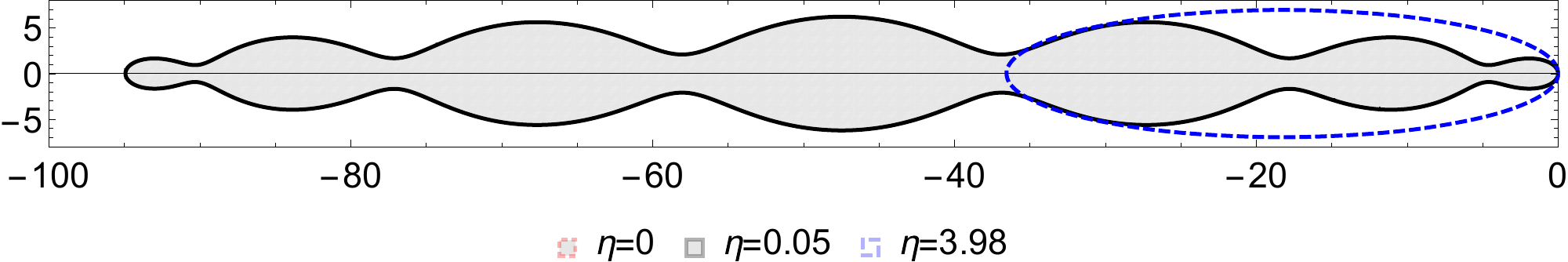}
			\caption{Complex stability domain for different damping parameters $\eta$.}
		\end{subfigure}\\[-1ex]
		\begin{subfigure}[t]{0.9\textwidth}
			\centering
			\includegraphics[width=0.9\linewidth]{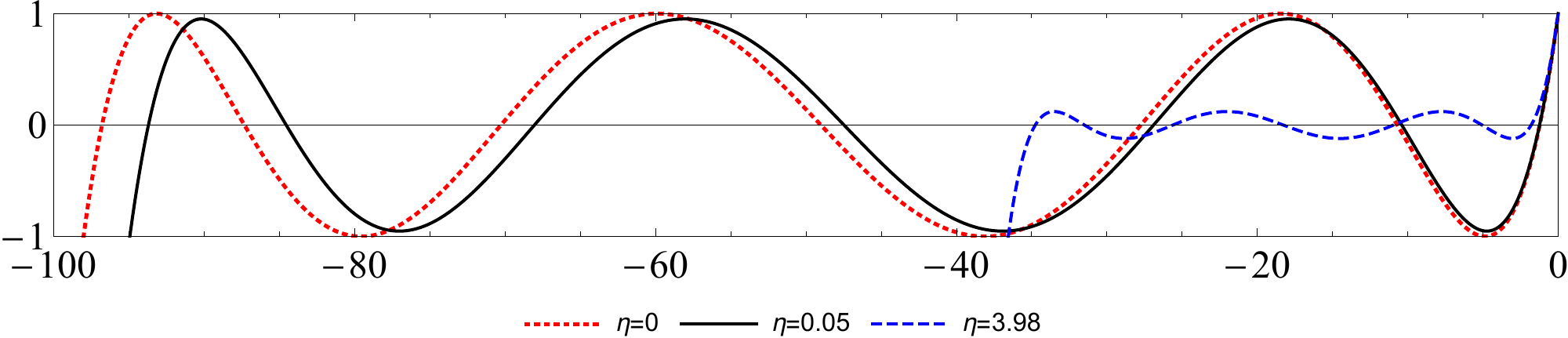}
			\caption{Corresponding stability functions $R_{7,\eta}(p)$.}
		\end{subfigure}
		\caption{
			Stability domains and stability functions of the deterministic Chebyshev method for $s=7$ and different damping values $\eta=0,0.05,3.98$.
			\label{fig:stabdetplot}}
	\end{figure}
	In Figure \ref{fig:stabdetplot}(a) we plot the complex stability domain $\{p\in \IC\,;\, |R_{s,\eta}(p)| \leq 1\}$
	for $s=7$ stages and different values $\eta=0$, $\eta=0.05$ and $\eta=3.98$, respectively.
	We also plot in Figure \ref{fig:stabdetplot}(b) the corresponding stability function $R_{s,\eta}(p)$ as a function of $p$ real, to illustrate that 
	the stability domain along the negative real axis corresponds to the values for which $|R_{s,\eta}(p)|\leq1$.
	We observe that in the absence of damping ($\eta=0$), the stability domain includes the large real interval $[-2\cdot s^2,0]$ of width $2\cdot 7^2=98$. However
	for all $p$ that are a local extrema of the stability function, where 
	$|R_{s,\eta}(p)|=1$, the stability domain is very thin and does not include a neighbourhood close to the negative real axis.
	To make the scheme robust with respect to small perturbations of the eigenvalues, it is therefore needed to add some damping 
	and a typical value is $\eta=0.05$, see for instance the reviews 
	\cite{Ver96b,Abd13c}.
	The advantage is that the stability domain now includes a neighbourhood of the negative real axis portion. The price of this improvement is a slight reduction
	of the stability domain size $C_\eta s^2$, where $C_\eta \simeq 2-\frac43 \eta$.
	 Chebyshev methods have been first generalized for It\^o SDEs in \cite{AbL08} (see \cite{AbC08} for Stratonovitch SDEs)  
	with the following stochastic orthogonal Runge-Kutta-Chebyshev method (S-ROCK):\footnote{A variant with analogous stability properties is proposed in \cite{AbL08}
	with $g^r(K_s)$ replaced by $g^r(K_{s-1})$ in \eqref{eq:stdmethod}.}
	\begin{eqnarray}
	\label{srock1}
	K_0&=&X_0 \nonumber\\
	K_1&=&X_0+\mu_1hf(X_0) \nonumber\\
	K_i&=&\mu_ihf(K_{i-1})+\nu_iK_{i-1}+\kappa_iK_{i-2},\quad i=2,\ldots,s, \nonumber\\
	X_1&=&K_s + \sum_{r=1}^mg^r(K_s)\Delta W_r, \label{eq:stdmethod}
	\end{eqnarray}
where the coefficients $\mu_i,\nu_i,\kappa_i$ are defined in \eqref{eq:omega},\eqref{coeff}. In contrast to the deterministic method \eqref{eq:Cheb1}, where $\eta$ is chosen small and fixed (typically $\eta=0.05$), in stochastic case for the classical S-ROCK method \cite{AbL08}, the damping $\eta=\eta_s$ is not small and chosen as an increasing function of $s$ that plays a crucial role in stabilizing the noise and in obtaining an increasing portion of the true stability domain \eqref{equ:estab_num_ell} as $s$ increases. 

\begin{figure}[t!b]
		\centering
		\begin{subfigure}[t]{0.9\textwidth} 
			\centering
			\includegraphics[scale=0.43]{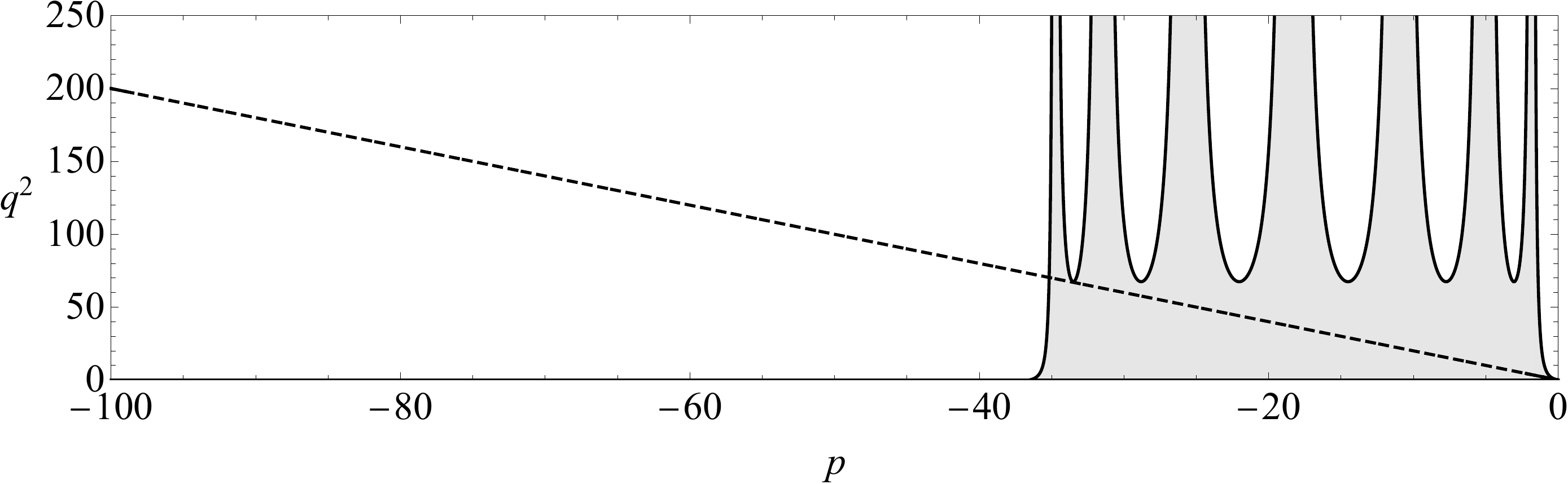}
			\caption{Standard S-ROCK method ($s=7$, $\eta=3.98$).}
		\end{subfigure}\\
		\begin{subfigure}[t]{0.9\textwidth}
			\centering
			\includegraphics[scale=0.43]{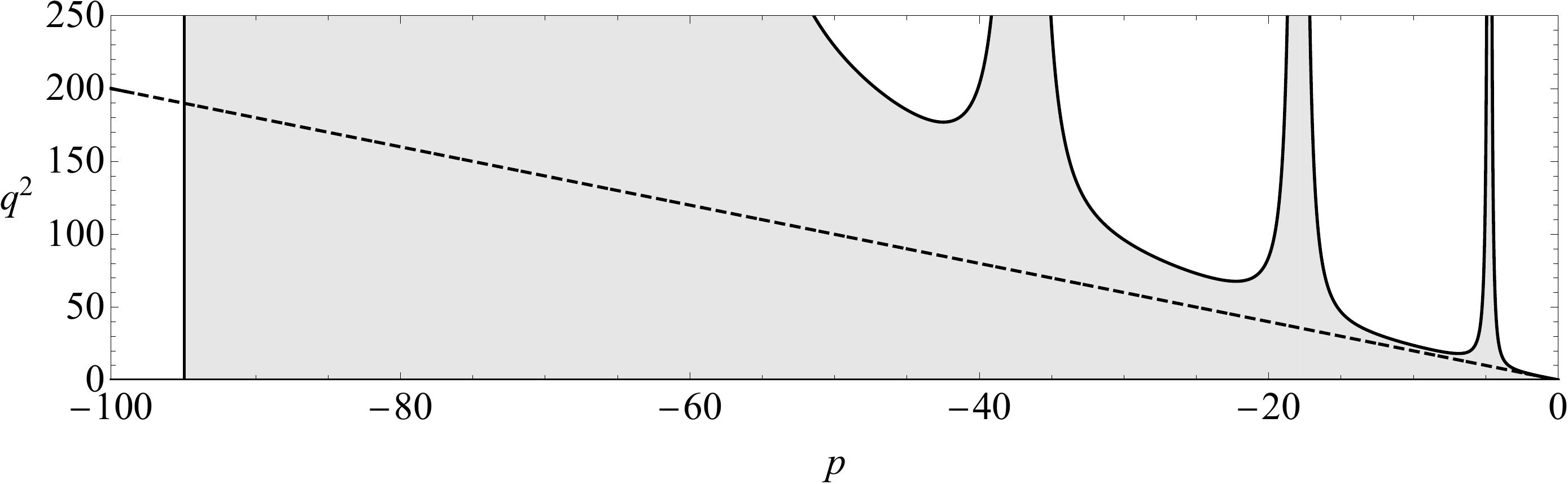}
			\caption{New SK-ROCK method ($s=7$, $\eta=0.05$).}
		\end{subfigure}\\
		\begin{subfigure}[t]{0.9\textwidth}
			\centering
			\includegraphics[scale=0.43]{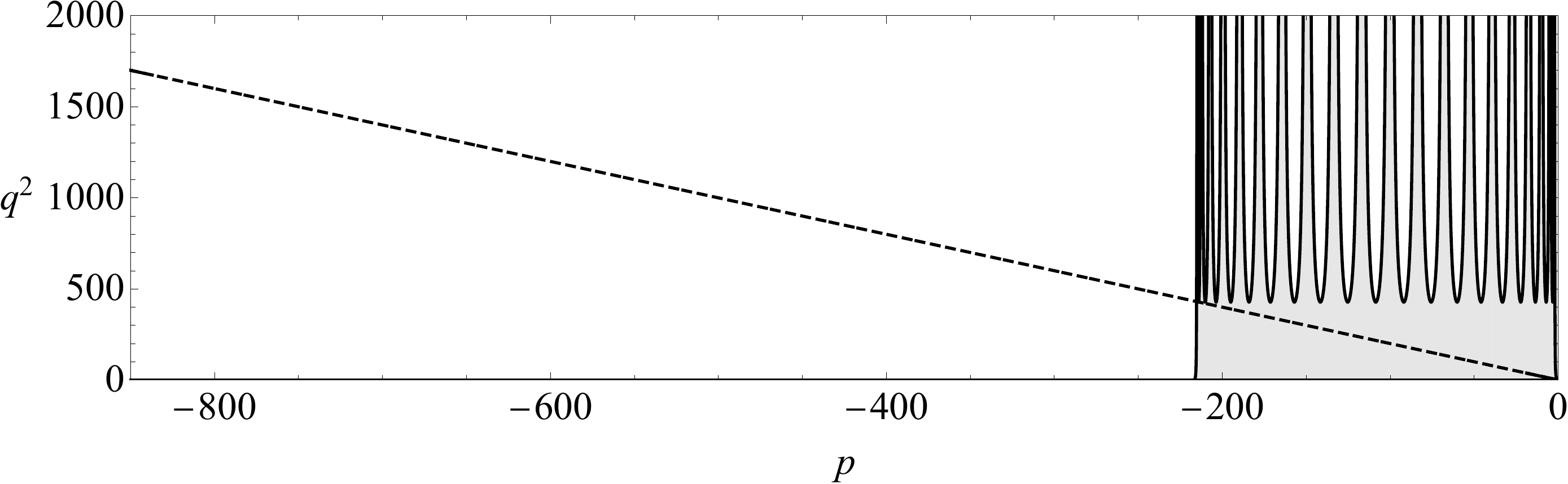}
			\caption{Standard S-ROCK method ($s=20$, $\eta=6.95$).}
		\end{subfigure}\\
		\begin{subfigure}[t]{0.9\textwidth}
			\centering
			\includegraphics[scale=0.43]{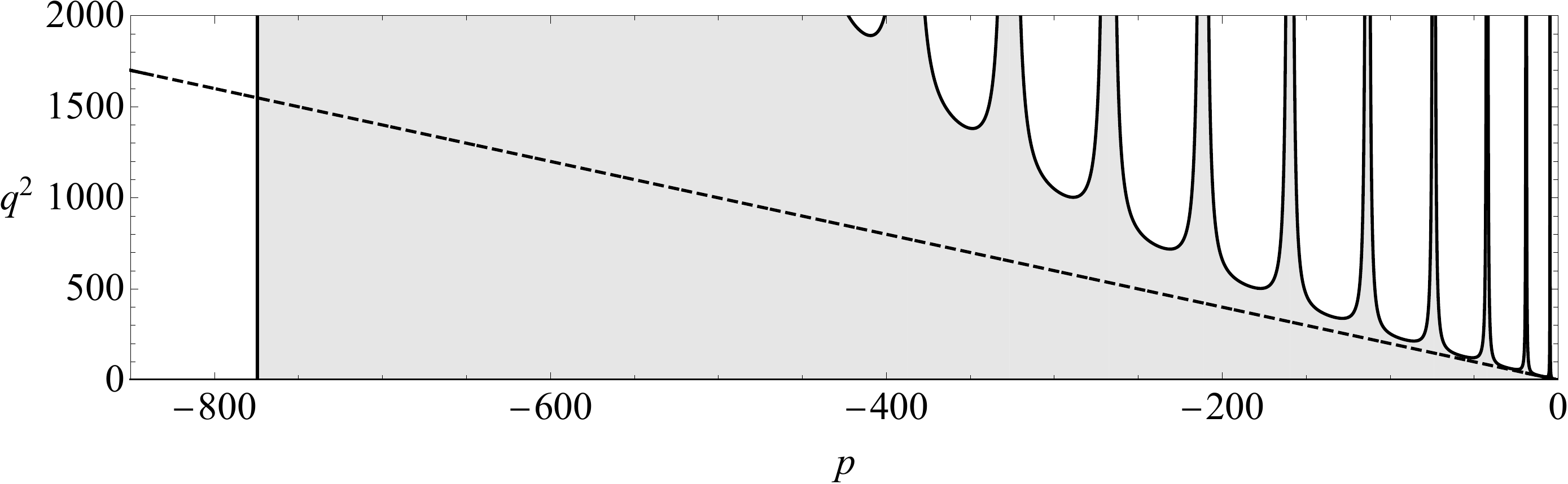}
			\caption{New SK-ROCK method ($s=20$, $\eta=0.05$).}
		\end{subfigure}
		\caption{
			Mean-square stability domains of the standard and new stochastic Chebyshev methods in the $p$--$q^2$ plane for $s=7,20$ stages, respectively. The dashed lines corresponds to the upper boundary $q^2=-2p$ of the real mean-square stability domain $\calS\cap\R^2$ of the exact solution.
			\label{fig:stabstoplot}}
	\end{figure}
	
	In the context of stiff SDEs, a relevant stability concept is that of mean-square stability.
	A test problem widely used in the literature is \cite{SaM96,Hig00,BBT04,Toc05} ,
	\begin{equation} \label{eq:lintest}
	dX(t)= \lambda X(t)dt + \mu X(t) dW(t),\qquad X(0)=1,
	\end{equation}
	in dimensions $d=m=1$ with fixed complex parameters $\lambda,\mu$.
	Note that other stability test problem in multiple dimensions are also be considered in \cite{BuC10} and references therein.
	The exact solution of \eqref{eq:lintest} is called mean-square stable if $\lim_{t\rightarrow\infty}{\mathbb E}\big(|X(t)|^2\big)=0$
	and this holds if and only if $(\lambda,\mu) \in \calS^{MS}$,
	where
	$$
	\calS^{MS} = \big \{(\lambda,\mu)\in\mathbb{C}^2\ ;\ \Re(\lambda)+\frac{1}{2}|\mu|^2<0\big\}.
	$$
	Indeed, the exact solution of \eqref{eq:lintest} is given by $X(t)=\exp((\lambda+\frac12\mu^2) t+\mu W(t))$, and an application of the the It\^o formula yields
	$
	\IE(|X(t)|^2)=\exp((\Re (\lambda)+\frac12\mu^2) t)
	$
	which tends to zero at infinity if and only if $\Re (\lambda)+\frac12\mu^2<0$.
	We say that a numerical scheme $\{X_n\}$ for the test problem \eqref{eq:lintest} is mean-square stable if and only if $\lim_{n\rightarrow \infty}\IE(|X_n|^2) =0$.
	For a one-step integrator applied to the test SDE \eqref{eq:lintest}, we obtain in general a induction of the form
	\begin{equation} \label{eq:difference}
	X_{n+1}= R(p,q,\xi_n)X_{n},
	\end{equation}
	where $p=\lambda h, q =\mu \sqrt{h}$, and $\xi_n$ is a random variable
	(e.g.
	a Gaussian
	$\xi_n\sim \mathcal{N}(0,1)$ 
	or a discrete random variable).
	Using $\IE(|X_{n+1}|^2)= \IE(|R(p,q,\xi_n)|^2) \IE(|X_{n}|^2),$ we obtain the mean-square stability condition
	\cite{SaM96,Hig00}
	\begin{equation} \label{equ:estab_num_dom}
	\lim_{n \rightarrow \infty} \IE(|X_{n}|^{2})=0
	\iff (p,q)\in {\calS_{num}},
	\end{equation} 
	where we define ${\mathcal S_{num}}=\left \{(p,q)\in\mathbb{C}^2\ ;  \IE|R(p,q,\xi)|^{2}<1 \right\}$.
	The function $R(p,q,\xi_n)$ is called the stability function of the one-step integrator.
	For instance, the stability function of the Euler-Maruyama method \eqref{eq:EM} reads $R(p,q,\xi)=1+p+q\xi$ and we have $\IE(|R(p,q,\xi)|^2)=(1+p)^2 + q^2.$
	
	We say that a numerical integrator is mean-square $A$-stable if $\calS^{MS} \subset \calS_{num}$.
	This means that the numerical scheme applied to \eqref{eq:lintest} is mean-square stable for all 
	all $h>0$ and all $(\lambda,\mu) \in \calS^{MS}$ for which the exact solution of \eqref{eq:lintest} is mean-square stable. An explicit Runge-Kutta type scheme cannot however be mean-square stable because its stability domain $\calS_{num}$ is necessary bounded along the $p$-axis.
	Following \cite{AbC08,AbL08}, we consider the following portion of the true mean-square stability domain
	\begin{equation}
	\label{equ:estab_num_ell}
	\mathcal{S}_{a}  = \{(p,q) \in (-a,0) \times \R\ ;\ p+\frac12 |q|^2 < 0\},
	\end{equation}
	and define for a given method 
	\begin{equation}
	\label{eq:defL}
	L=\sup\{a>0\ ;\ {\mathcal{S}_{a} \subset \mathcal{S}_{num}  } \}.
	\end{equation}
	We search for explicit schemes for which the length $L$ of the stability domain is large.
	For example, for the classical S-ROCK method \cite{AbL08}, the value $\eta=3.98$ is the optimal damping maximising $L$ for $s=7$ stages and we can see in Figure \ref{fig:stabdetplot} that this damping reduces significantly the stability domain compared to the optimal deterministic domain.


	The new S-ROCK method, denoted SK-ROCK   (for stochastic second kind orthogonal Runge-Kutta-Chebyshev method) introduced in this paper is defined as
	\begin{eqnarray}
	K_0&=&X_0 \nonumber\\
	K_1&=&X_0+\mu_1hf(X_0+\nu_1 Q) +\kappa_1Q \nonumber\\
	K_i&=&\mu_ihf(K_{i-1})+\nu_iK_{i-1}+\kappa_iK_{i-2},\quad i=2,\ldots,s.\nonumber\\
	X_1&=&K_s, \label{eq:newmethod}
	\end{eqnarray}
	where
	$
	Q=\sum_{r=1}^mg^r(X_0)\Delta W_r, 
	$
	and $\mu_1=\omega_1/\omega_0,\nu_1={s \omega_1}/{2},\kappa_1={s \omega_1}/{\omega_0}$ and $\mu_i,\nu_i,\kappa_i,~i=2,\ldots,s$ are given by \eqref{coeff}, with a fixed small damping parameter $\eta$.
	In the absence of noise ($g^r=0,r=1,\ldots, m$, deterministic case), this method coincides with the standard deterministic order $1$ Chebychev method, see the review \cite{Abd13c}.
	We observe that the new class of methods \eqref{eq:newmethod} is closely related to the
	standard S-ROCK method \eqref{srock1}.
	Comparing the two schemes \eqref{eq:newmethod} and \eqref{eq:stdmethod},
	the two differences are on the one hand that the noise term is computed at the first internal stage $K_1$ for \eqref{eq:newmethod}, whereas it is computed at the final stage in \eqref{eq:stdmethod}, and on the other hand, for the new method \eqref{eq:newmethod} the damping parameter $\eta$ involved in \eqref{eq:omega}
	is fixed and small independently of $s$ (typically $\eta=0.05$), whereas for the standard method \eqref{eq:stdmethod}, the damping $\eta$ is an increasing function of $s$, optimized numerically for each number of stages $s$.

	If we apply the above scheme \eqref{eq:newmethod} to  the linear test equation \eqref{eq:lintest}, 
we obtain 	$$
		X_{n+1} = R(p,q,\xi_n) X_n, 
		$$
where
 \begin{equation} \label{eq:RAB}
		\IE(|R(p,q,\xi)|^2) = A(p)^2 + B(p)^2 q^2,
		\end{equation}
		and
		$$
		A(p)=\frac{T_s(\omega_0+\omega_1 p)}{T_s(\omega_0)} \qquad
		B(p)=\frac{U_{s-1}(\omega_0+\omega_1 p)}{U_{s-1}(\omega_0)}(1+ \frac{\omega_1}2 p)
		$$
		correspond to the drift and diffusion contributions, respectively. 
The above stability function  (see Lemma \ref{lemma:stabfunction} in Section \ref{sec:analysis}) is obtained by using the recurrence relation for the first kind Chebyshev polynomials \eqref{eq:recTU1} and the similar recurrence relation for the second kind Chebyshev polynomials
	\begin{equation} \label{eq:recTU2}
		U_j(p) = 2p U_{j-1}(p) - U_{j-2}(p),\qquad
		\end{equation}
		where $U_0(p)= 1,U_1(p)=2p$.
		Notice that the relation $T_s'(p)=s U_{s-1}(p)$ between first and second kind Chebyshev polynomials will be repeatedly used in our analysis.

	In Figure \ref{fig:stabstoplot}(b)(d), we plot the mean-square stability domain of the SK-ROCK method for $s=7$ and $s=20$ stages, respectively and the same small damping $\eta=0.05$ as for the deterministic Chebyshev method. We observe that the stability domain has length $L_s\simeq (2-\frac43 \eta)s^2$.
	For comparison, we also include in Figure \ref{fig:stabstoplot}(a)(c) the mean-square stability domain of the standard S-ROCK method with smaller stability domain size $L_s\simeq 0.33\cdot s^2$.

In Figure \ref{fig:RAB}, we plot the stability function $\IE(|R(p,q,\xi)|^2)$ in \eqref{eq:RAB} as a function
		of $p$ for various scaling of the noise for $s=7$ stages and damping $\eta=0.05$. 
		We see that it is bounded by $1$ for $p\in(-2(1-\frac23 \eta)s^2,0)$ which is proved asymptotically in Theorem \ref{thm:stabsize}.
		The case $q=0$ corresponds to the deterministic case, and we see in Figure \ref{fig:RAB}(a), the polynomial $\IE(|R(p,0,\xi)|^2)=A(p)^2$.
		Noticing that $\IE(|R(p,q,\xi)|^2)$ is an increasing function of $q$, the case $q^2=-2p$ represented in Figure \ref{fig:RAB}(c) corresponds to the upper border of the stability domain $\calS_L$ defined in \eqref{equ:estab_num_ell} (note that this is the stability function value along the dashed boundary in Figure \ref{fig:stabstoplot}), while the scaling $q^2=-p$ in Figure \ref{fig:RAB}(c) is an intermediate regime. In Figures \ref{fig:RAB}(b)(c), we also include the drift function $A(p)^2$ (red dotted lines) and diffusion function $B(p)^2 q^2$ (blue dashed lines), and it can be observed that their oscillations alternate, which means that any local maxima of one function is close to a zero of the other function.
		This is not surprising because $A(p)$ and $B(p)$ are related to the first kind and second kind Chebyshev polynomials, respectively, corresponding to the cosine and sine functions.
		This also explains how a large mean-square stability domain can be achieved by the new SK-ROCK method \eqref{eq:newmethod} with a small damping parameter $\eta$, in contrast to the standard S-ROCK method \eqref{eq:stdmethod} from \cite{AbL08} that uses a large and $s$-dependent damping parameter $\eta$ with smaller stability domain size $L_s\simeq 0.33\cdot s^2$(see Figures \ref{fig:RAB}(a)(c)).
%
%
\begin{figure}[tb]
		\centering
		\begin{subfigure}[t]{0.9\textwidth}
			\centering
			$\IE(|R(p,q,\xi)|^2)={\color{red} A(p)^2}+{\color{blue} B(p)^2q^2}$\\
			\includegraphics[width=0.9\linewidth]{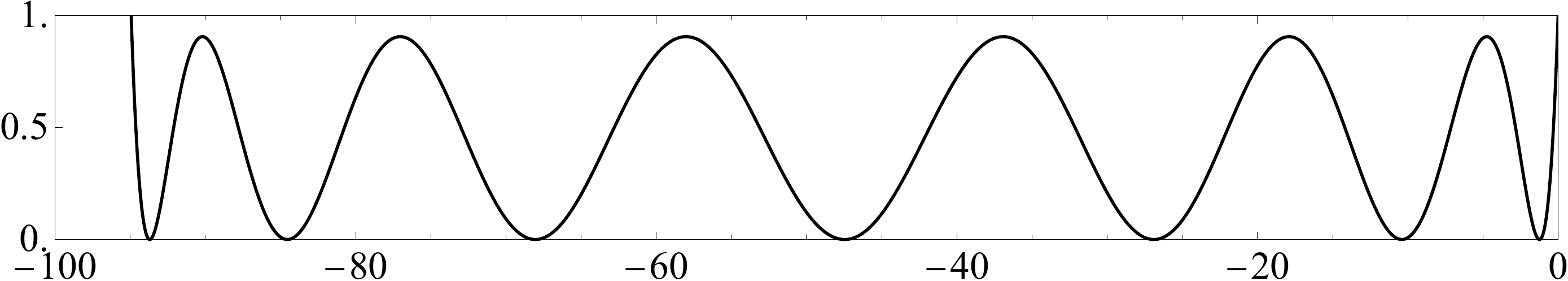}
			\caption{$q=0$.}
		\end{subfigure}\\[-1ex]
		\begin{subfigure}[t]{0.9\textwidth}
			\centering
			\includegraphics[width=0.9\linewidth]{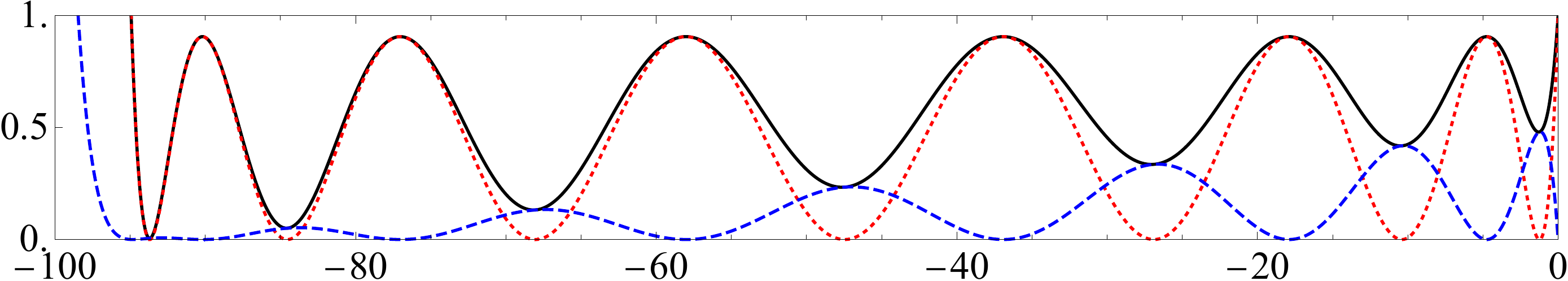}
			\caption{$q^2=-p$.}
		\end{subfigure}\\[-1ex]
		\begin{subfigure}[t]{0.9\textwidth}
			\centering
			\includegraphics[width=0.9\linewidth]{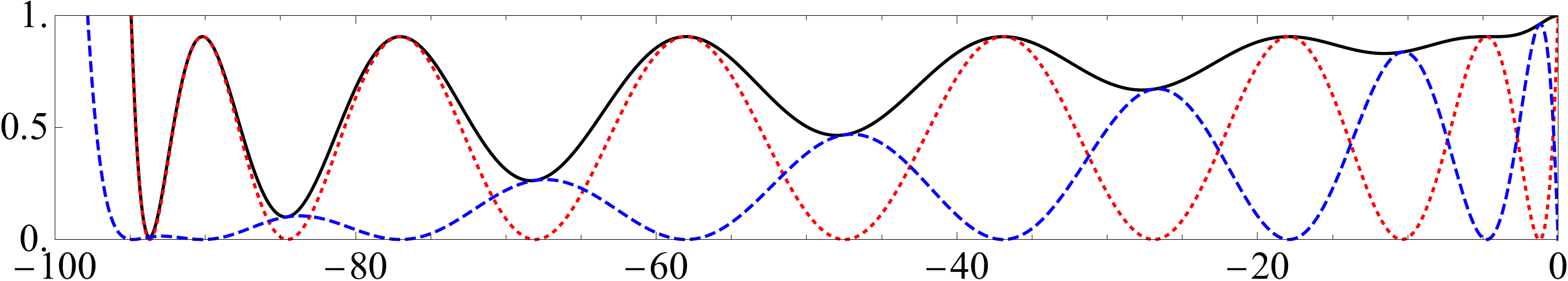}
			\caption{$q^2=-2p$.}
		\end{subfigure}
		\caption{ \label{fig:RAB}
			Stability function $\IE(|R(p,q,\xi)|^2)=A(p)^2+B(p)^2q^2$ as a function of $p$ in \eqref{eq:RAB} (solid black lines) of the new SK-ROCK method, $s=7,\eta=0.05$, for various noise scalings $q^2=0,-p,-2p$, respectively.
			We also include the drift contribution $A(p)^2$ (red dotted lines) and diffusion contribution $B(p)^2 q^2$ (blue dashed lines).}
	\end{figure}

	\section{Mean-square stability analysis}
	\label{sec:analysis}
	
	In this section, we prove asymptotically that the new SK-ROCK methods have an extended mean-square stability domain with size $Cs^2$	growing quadratically as a function of the number of internal stages $s$, where the constant $C \geq 2-\frac43\eta$ is the same as the optimal constant of the standard Chebyshev method in the deterministic case, using a fixed and small damping parameter $\eta$.
	
	\begin{lemma} \label{lemma:stabfunction} 
		Let $s\geq 1$ and $\eta\geq 0$.
		Applied to the linear test equation $dX=\lambda X dt + \mu X dW$, the scheme \eqref{eq:newmethod} yields
		$$
		X_{n+1} = R(\lambda h,\mu \sqrt h,\xi_n) X_n 
		$$
		where $p=\lambda h, q=\mu \sqrt h$,  $\xi_n \sim \calN(0,1)$ is a Gaussian variable and the stability function given by
		\begin{equation} \label{eq:defR}
		R(p,q,\xi) = \frac{T_s(\omega_0+\omega_1 p)}{T_s(\omega_0)} 
		+ \frac{U_{s-1}(\omega_0+\omega_1 p)}{U_{s-1}(\omega_0)}(1+ \frac{\omega_1}2 p) q \xi.
		\end{equation}
	\end{lemma}
	\begin{proof}
	Indeed, we take advantage that $T_j$ and $U_j$ have the same recurrence relations \eqref{eq:recTU1},\eqref{eq:recTU2},	and only the initialization changes with $T_1(x)= x$ and $U_1(x)= 2x$,
		we deduce $Q=X_0 \mu\sqrt{h}\xi$, and we obtain by induction on $i\geq 1$,
		$$
		K_i=\frac{T_i(\omega_0+\omega_1p)}{T_i(\omega_0)}X_0+\frac{U_{i-1}(\omega_0+\omega_1p)}{T_i(\omega_0)} (1+\frac{\omega_1p}{2})s \omega_1Q
		$$
		and we use $T_s'(x)=xU_{s-1}(x)$ and $s\omega_1/T_s(\omega_0)=1/U_{s-1}(\omega_0)$, which yields the result for $X_1=K_s$.
	\end{proof}

	For a positive damping $\eta$, we prove the following main result of this section, showing that a quadratic growth $L\geq (2-4/3\, \eta)s^2$ of the mean-square stability domain defined in \eqref{eq:defL} is achieved for all 	$\eta$ small enough and all stage number $s$ large enough.
	
	\begin{theorem} \label{thm:stabsize}
		There exists $\eta_0>0$ and $s_0$ such that for all $\eta \in [0,\eta_0]$ and all $s\geq s_0$,
		for all $p\in [-2\omega_1^{-1},0]$ and $p + \frac 12|q|^2 \leq 0$, we have $\IE(|R(p,q,\xi)|^2) \leq 1$.
	\end{theorem}
	
		\begin{remark} \label{rem:exactetazero}
		We deduce from Theorem \ref{thm:stabsize}, that the mean-square stability domain size 
		\eqref{eq:defL} of SK-ROCK
		grows as $(2-4/3\, \eta)s^2$ which is arbitrarily close to the optimal stability domain size $2s^2$ for $\eta\rightarrow 0$.
		Indeed, for $s\rightarrow \infty$ and all $\eta\leq \eta_0$, we have $2\omega_1^{-1}s^{-2} \rightarrow 2\frac{\tanh(\sqrt{2\eta})}{\sqrt{2\eta}} = 2-4/3\, \eta + \bigo(\eta^2)$ and for all $s,\eta$, we have $2\omega_1^{-1}\geq (2-4/3\,\eta) s^2$.
		In addition, in the special case of a zero damping ($\eta=0$), 
		the stability function \eqref{eq:defR} reduces to
		$$
		R(p,q,\xi) = T_s(1+\frac p{s^2}) + s^{-1}U_{s-1}(1+\frac p{s^2})(1+\frac p{2s^2}) q \xi,
		$$
		and it holds 
		$$
		\IE(|R(p,q, \xi)|^2) \leq 1,
		$$
		for all $s\geq1$, for all $p\in[-2s^2,0]$ and all $q \in \IC$ such that $p+|q|^2/2 \leq 0$.
		Indeed, for $p\in[-2s^2,0]$, we denote $\cos\theta = 1+\frac p{s^2} \in[-1,1]$ and using $$T_s(\cos(\theta))=\cos(s\theta),\qquad \sin(\theta) U_{s-1}(\cos(\theta))=\sin(s\theta),$$ we obtain
		$$
		\IE(|R(p,q, \xi)|^2) \leq \IE(|R(p,\sqrt{-2p})  = \cos(s\theta)^2+ \sin(s\theta)^2 \frac{1+\cos \theta}2 \leq 1,
		$$
		where we used $-2p=2s^2(1-\cos\theta)$, $1+\frac p{2s^2} = \frac{1+\cos\theta}2$ and $\sin^2 \theta = (1+\cos \theta)(1-\cos\theta)$.
	\end{remark}

	Before we prove Theorem \ref{thm:stabsize}, we have the following lemma, 
	see \cite{HoS80} for analogous results.
	\begin{lemma}\label{limits}
		We have the following convergences as $s\rightarrow \infty$ to analytic functions\footnote{%
			Note that for $z<0$, we can use $\sqrt{2z}=i\sqrt{-2z}$ and obtain $T_s(1+z/s^{2})\rightarrow \cos(\sqrt{-2z})$ for $s\rightarrow \infty$ and similarly $\alpha(z)=\sinc(\sqrt{-2z})$.} 
		uniformly for $z$ in any bounded set of the complex plan,
		\begin{align*}
		T_s(1+z/s^{2}) &\to \cosh{\sqrt{2z}},\\
		s^{-1}U_{s-1}(1+z/s^{2})&\to\alpha(z):= \frac{\sinh{\sqrt{2z}}}{\sqrt{2z}},\\
		\omega_1s^{2}&\to \Omega(\eta)^{-1}, \quad \Omega(\eta):=\frac{\tanh{\sqrt{2\eta}}}{\sqrt{2\eta}}. 
		\end{align*}
	\end{lemma}
	\begin{proof}
		We prove the uniform convergence of the first limit only, since it will be useful in the proof of the next theorem. The others can be proved in a similar way.
		
		First, let us write the two functions of $\eta$ in Taylor series,
		\begin{eqnarray*}
			s^{-1}U_{s-1}(\omega_0)&=&
			s^{-1}\sum_{n=0}^{s-1}\frac{U_{s-1}^{(n)}(1)}{n!}(\frac\eta{s^2})^{n}
			=\sum_{n=0}^{s-1}\left( \frac1{n!}\prod_{k=0}^{n}(1-\frac{k^2}{s^2})\prod_{k=0}^{n}\frac{1}{2k+1}\right)\eta^n,\\
						\alpha(\eta)&=&\sum_{n=0}^{\infty}\frac{2^n\eta^n}{(2n+1)!}
			=\sum_{n=0}^{\infty}\left(\frac{1}{n!}\prod_{k=0}^{n}\frac{1}{2k+1}\right)\eta^n,
		\end{eqnarray*}	
		where we used the formula
$sU_{s-1}^{(n-1)}(1)=T_s^{(n)}(1)=\prod_{k=0}^{n-1}\frac{s^2-k^2}{2k+1}$.
Subtracting the above two identities, we deduce
		\begin{eqnarray} \label{eq:subdiff}
			\sup_{\eta\in[-\eta_0,\eta_0]}\left|s^{-1}U_{s-1}(\omega_0)-\alpha(\eta)\right|
			&\leq& \sum_{n=0}^{s-1}\frac{\eta_0^n}{n!}\left(1-\prod_{k=0}^{n}(1-\frac{k^2}{s^2})\right)\prod_{k=0}^{n}\frac{1}{2k+1}
			+\sum_{n=s}^{\infty}\frac{\eta_0^n}{n!} \nonumber\\
			&\leq&\sum_{n=0}^{s-1}\frac{\eta_0^n}{n!}\left(1-\prod_{k=0}^{n}(1-\frac{k^2}{s^2})\right)\frac{1}{2s-1}
			+\sum_{n=s}^{\infty}\frac{\eta_0^n}{n!}
		\end{eqnarray}
		Noticing that $\frac{\eta_0^n}{n!}\big(1-\prod_{k=0}^{n}(1-\frac{k^2}{s^2})\big)\frac{1}{2s-1}$ converges to zero as $s\rightarrow \infty$ and is bounded
		by $\frac{\eta_0^n}{n!}$ for all integers $s,n$, which is the general term of the convergent series of $\exp(\eta_0)=\sum_{n=0}^{\infty}\frac{\eta_0^n}{n!}$, the Lebesgue dominated convergence theorem implies that \eqref{eq:subdiff}
		converges to zero as $s\rightarrow \infty$, which concludes the proof.
		%
	\end{proof}

	\begin{lemma}\label{est}
		For all $\eta$ small enough and all $s$ large enough, we have the following estimate:
		\begin{equation}\label{estimate}
		\frac{s^2 \omega_1}{T_s(w_0)^2} \frac{1-(1-\omega_1)^2}{1-(\omega_0-\omega_1)^2}  \leq 1
		\end{equation}
	\end{lemma}
	\begin{proof}[Proof of Lemma \ref{est}]
		Using the Lemma \ref{limits} we have for $s\to\infty$, uniformly for all $\eta\in[0,\eta_0]$,
		\begin{equation*}
		\frac{s^2\omega_1}{T_s(\omega_0)^2}\to\frac{2\sqrt{2\eta}}{\sinh(2\sqrt{2\eta})}~and~\frac{1-(1-\omega_1)^2}{1-(\omega_0-\omega_1)^2}\to\frac{1}{1-\Omega(\eta)\eta}.
		\end{equation*}
		Now if we expand both functions in Taylor series we get:
		\begin{equation}
		\frac{2\sqrt{2\eta}}{\sinh(2\sqrt{2\eta})}=1-\frac43\eta+\bigo(\eta^2),\qquad\frac{1}{1-\Omega(\eta)\eta}=1+\eta+\bigo(\eta^2),
		\end{equation}
		and this implies that for all $s$ large enough and all $\eta\leq \eta_0$,
		\begin{equation}
		\frac{s^2 \omega_1}{T_s(w_0)^2} \frac{1-(1-\omega_1)^2}{1-(\omega_0-\omega_1)^2}\leq (1-\frac43\eta_0+\bigo(\eta_0^2))(1+\eta_0+\bigo(\eta_0^2))=1-\frac13\eta_0+\bigo(\eta_0^2),
		\end{equation}
		which is less than 1 for $\eta_0$ small enough.
	\end{proof}
	
	\begin{remark}
	Numerical evidence suggests that the result of Theorem \ref{thm:stabsize} holds for all $s\geq 1$ and all $\eta\geq 0$.
		Indeed, it can be checked numerically that \eqref{estimate} holds for all $\eta\in(0,1)$ and all $s\geq1$.
	\end{remark}
	
	\begin{proof}[Proof of Theorem \ref{thm:stabsize}.]
		Setting $x=w_0+w_1p$, a calculation yields
		\begin{eqnarray*}
			\IE(|R(p,q,\xi)|^2) &\leq &  \IE(|R(p,\sqrt{-2p},\xi)|^2)\\
			&=& \frac{T_s(x)^2}{T_s(w_0)^2} + \frac{U_{s-1}(x)^2}{U_{s-1}(w_0)^2} (1+\frac {w_1}{2}p)^2 (-2p)
		\end{eqnarray*}
		The proof is conducted in two steps, where we treat separately the cases $p\in [-2\omega_1^{-1},-1]$ and $p\in[-1,0]$. 
		For the first case $p\in [-2\omega_1^{-1},-1]$, which corresponds to $x\in [-1+\eta/s^2,\omega_0-\omega_1]$, we have
		\begin{eqnarray*}
			\IE(|R(p,q,\xi)|^2)  &=& \frac{T_s(x)^2}{T_s(w_0)^2} + \frac{U_{s-1}(x)^2}{U_{s-1}(w_0)^2} \left(1-\frac {w_0-x}{2}\right)^2 2\frac{w_0-x}{w_1}\\
			&=& \frac{T_s(x)^2}{T_s(w_0)^2} + U_{s-1}(x)^2(1-x^2) Q_s(x)
		\end{eqnarray*}
		where we denote
		\begin{eqnarray*}
			Q_s(x) &=& 
			\frac{s^2 \omega_1}{T_s(w_0)^2}  \left(\frac {1+x-\frac{\eta}{s^2}}{2}\right) \frac{1-(x-\frac{\eta}{s^2})^2}{1-x^2}
		\end{eqnarray*}
		First, we note that 
		$\frac {1+x-\frac{\eta}{s^2}}{2} \in [0,1-\frac{\omega_1}2]$.
		Next, using $\frac{\eta}{s^2} \leq 2$, we deduce $\frac{d}{dx}\left( \frac{1-(x-\frac{\eta}{s^2})^2}{1-x^2} \right) = \frac{2\eta}{s^2} \frac{1+x^2-\eta/s^2 x}{(1-x^2)^2}\geq 
		\frac{2\eta}{s^2} \frac{(1-x)^2}{(1-x^2)^2} \geq 0$.
		Thus, $\frac{1-(x-\frac{\eta}{s^2})^2}{1-x^2}$ is an increasing function of $x$, smaller than its value at $x=\omega_0-\omega_1$,
		$$
		\frac{1-(x-\frac{\eta}{s^2})^2}{1-x^2} \leq \frac{1-(1-\omega_1)^2}{1-(\omega_0-\omega_1)^2} 
		$$
		Using Lemma \ref{est} we obtain $|Q_s(x)| \leq 1$. This yields $\IE(|R(p,q,\xi)|^2) \leq T_s(x)^2 + U_{s-1}(x)^2(1-x^2) =1$.
		
		For the second case $p\in[-1,0]$ which corresponds to $x\in [\omega_0-\omega_1,\omega_0]$, we deduce from $T_s(x)^2 + U_{s-1}(x)^2(1-x^2) =1$ that
		\begin{eqnarray*}
			\IE(|R(p,q,\xi)|^2)
			&\leq&  \frac1{T_s(w_0)^2} + \frac{U_{s-1}(x)^2}{U_{s-1}(w_0)^2} \left( (1+\frac {w_1}{2}p)^2 (-2p)-\frac{(1-x^2)U_{s-1}(\omega_0)^2}{T_s(w_0)^2}\right)
		\end{eqnarray*}
		Using Lemma \ref{limits}, we get 
		\begin{eqnarray*}
			\IE(|R(p,q,\xi)|^2)
			&\leq&  \frac1{T_s(w_0)^2} + \frac{U_{s-1}(x)^2}{U_{s-1}(w_0)^2} \left( (1+\frac {w_1}{2}p)^2 (-2p)-\frac{(1-x^2)U_{s-1}(\omega_0)^2}{T_s(w_0)^2}\right)\\
			&\to&l(\eta,p):=\frac1{\cosh^2{\sqrt{2\eta}}}+\frac{\alpha(\eta+p/\Omega(\eta))^2}{\alpha(\eta)^2}(-2p(\Omega(\eta)-1)+2\Omega(\eta)^2\eta),
		\end{eqnarray*}
		for $s\to\infty$, where the above convergence is uniform for $p\in[0,1]$,$\eta\leq \eta_0$.
		Using the fact that $\Omega(\eta)=1-\frac23\eta+O(\eta^2)$, we deduce 
		 $$\frac{\partial l}{\partial \eta}|_{\eta=0}=-2+\alpha(p)^2(-\frac{4}{3}p+2).$$
		By Taylor series in the neighbourhood of zero we have 
		$\alpha(p)^2=1+\frac{2}{3}p+\frac{8}{45}p^2+O(p^3)$,
		and for $p\in[-1,0],$ $\alpha(p)^2\leq 1+\frac{2}{3}p+\frac{8}{45}p^2$, thus
		for all $p\in[-1,0]$,
		\begin{equation*}
			\frac{\partial l}{\partial \eta}|_{\eta=0}\leq-2+(1+\frac{2}{3}p+\frac{8}{45}p^2)(-\frac{4}{3}p+2)
			=-\frac{8}{135}p^2(4p+9)\leq0.
		\end{equation*}
		Therefore, there exists $\eta_0$ small enough such that for all $p\in[-1,0],\eta\leq\eta_0$, $l(\eta,p)\leq l(0,p)=1$.
		This concludes the proof of Theorem \ref{thm:stabsize}.
	\end{proof}

	\section{Convergence analysis}
	
	\label{sec:convergence}
	
	We show in this section that the proposed scheme \eqref{eq:newmethod} has strong order 1/2 and weak order $1$
	for general systems of SDEs of the form \eqref{eq:sde0} with Lipschitz and smooth vector fields, analogously to the simplest Euler-Maruyama method.
	
	We denote by $C_P^4(\IR^d,\IR^d)$ the set of functions from $\IR^d$ to $\IR^d$ that are $4$ times continuously differentiable
	with all derivatives with at most polynomial growth.
The following theorem shows that the proposed SK-ROCK has strong order 1/2 and weak order $1$ for general SDEs.
	\begin{theorem} \label{thm:conv}
		Consider the system of SDEs \eqref{eq:sde0} on a time interval of length $T>0$, with $f,g \in C_P^4(\IR^d,\IR^d)$, Lipschitz continuous.
		Then the scheme \eqref{eq:newmethod} has strong order $1/2$ and weak order $1$,
		\begin{align}
		\IE|(\| X(t_n) - X_n \|)  &\leq Ch^{1/2},\qquad t_n=nh\leq T, \label{eq:strongconv}\\
		| \IE(\phi(X(t_n))) - \IE(\phi(X_n)) |  &\leq Ch,\qquad t_n=nh\leq T, \label{eq:weakconv}
		\end{align}
		for all $\phi\in C_P^4(\IR^d,\IR)$, where $C$ is independent of $n,h$.
	\end{theorem}
For the proof the Theorem \ref {thm:conv}, the following lemma will be useful.
		It relies on the linear stability analysis of Lemma \ref{lemma:stabfunction}.
	\begin{lemma} \label{lemma:localerror}
		The scheme \eqref{eq:newmethod} has the following Taylor expansion after one timestep,
		$$
		X_1 = X_0 + hf(X_0) + \sum_{r=1}^mg^r(X_0)\Delta W_r + h \left(\frac{T_s''(\omega_0)\omega_1^2}{T_s(\omega_0)} + \frac{\omega_1}2\right) f'(X_0) \sum_{r=1}^mg^r(X_0)\Delta W_r + h^2 R_h(X_0),
		$$
		where all the moments of $R_h(X_0)$ are bounded uniformly with respect to $h$ assumed small enough, with a polynomial growth with respect to $X_0$.
	\end{lemma}
	\begin{proof}
		Using the definition \eqref{eq:newmethod} of the scheme and the recurrence relations \eqref{eq:recTU1},\eqref{eq:recTU2}, we obtain by induction on $i=1,\ldots,s$,
		\begin{align} \label{eq:K_i}
		K_i &= X_0 + h \frac{T_i'(\omega_0)\omega_1}{T_i(\omega_0)} f(X_0) + \frac{sT_i'(\omega_0)\omega_1}{iT_i(\omega_0)}\sum_{r=1}^mg^r(X_0)\Delta W_r \nonumber \\
		&+ h \left(\frac{sT_i''(\omega_0)\omega_1^2}{iT_i(\omega_0)} + \frac{sT_i'(\omega_0)\omega_1^2}{2iT_i(\omega_0)}\right) f'(X_0) \sum_{r=1}^mg^r(X_0)\Delta W_r + h^2 R_{i,h}(X_0),
		\end{align}
		and $R_{i,h}(X_0)$ has the properties claimed on $R_{h}(X_0)$. Using $\omega_1 = T_s(\omega_0)/T_s'(\omega_0)$, this yields the result for $X_1=K_s$.
	\end{proof}

	\begin{proof}[Proof of Theorem \ref{thm:conv}]
		A well-known theorem of Milstein \cite{Mil86} (see \cite[Chap.\ts 2.2]{MiT04})
		allows to infer the global orders of convergence from the error after one step.
		We first show that for all $r\in \mathbb{N}$ the moments $\IE(|X_n|^{2r})$ are bounded for all $n,h$ with $0\leq nh \leq T$ uniformly
		with respect to all $h$ sufficiently small.
		Then, it is sufficient to show the local error estimate
		$$
		|\IE(\phi(X(t_1))) -\IE(\phi(X_1))| \leq C h^{2},
		$$
		for all initial value $X(0)=X_0$ and where $C$ has at most polynomial growth with respect to $X_0$, to deduce the weak convergence estimate \eqref{eq:weakconv}.
		For the strong convergence \eqref{eq:weakconv}, using the classical result from \cite{Mil87}, it is sufficient to show in addition the local error estimate
		$$
		\IE(\|X(t_1) - X_1\|) \leq C h
		$$
		for all initial value $X(0)=X_0$ and where $C$ has at most polynomial growth with respect to $X_0$. 
		These later two local estimates are an immediate consequence of Lemma \ref{lemma:localerror}.
		
		To conclude the proof of the global error estimates, it remains to check that for all $r\in\mathbb{N}$ the moments $\IE(|X_n|^{2r})$ are bounded  uniformly with respect to all $h$ small enough for all $0\leq nh \leq T$. 
		We use here the approach of \cite[Lemma 2.2, p.\ts 102]{MiT04} which states that
		it is sufficient to show
		\begin{equation} \label{eq:lemma_moments}
		|\IE(X_{n+1}-X_n| X_n)| \leq C (1+| X_n|) h, \qquad | X_{n+1}- X_n| \leq M_n (1+| X_n|) \sqrt h,
		\end{equation}
		where $C$ is independent of $h$ and $M_n$ is a random variable with moments of all orders bounded uniformly
		with respect to all $h$ small enough.
		These estimates are a straightforward consequence of the definition 
		\eqref{eq:newmethod} of the scheme and the linear growth of $f,g$ (a consequence of their  Lipschitzness). 
		This concludes the proof of Theorem \ref{thm:conv}.
	\end{proof}
	\begin{remark}
		In the case of additive noise, i.e. $g^r,r=1,\ldots,m$ are constant functions, one can show that the 
		order of strong convergence \eqref{eq:strongconv} become $1$, analogously to the case of the Euler-Maruyama method.
		For a general multiplicative noise, a scheme of strong order one can also be constructed with $\IE(|R(p,q,\xi)|^2) \leq 1$
		for all $p\in[-2\omega_1^{-1},0]$ and all $q$ with $p+\frac{|q|^2}2 \leq 0$, as it can be check numerically.
		The idea is to modify the first stages of the scheme such that the stability function \eqref{eq:defR}   becomes
		$$
		R(p,q,\xi) = \frac{T_s(\omega_0+\omega_1 p)}{T_s(\omega_0)} 
		+ \frac{U_{s-1}(\omega_0+\omega_1 p)^2}{U_{s-1}(\omega_0)^2}(1+ \frac{w_1}2 p-\frac{\omega_1^4}2 p^2)\left(q \xi+q^2\frac{\xi^2-1}{2}\right).
		$$
		We refer to \cite[Remark 3.2]{AVZ13} for details.
	\end{remark}

	\section{Long term accuracy for Brownian dynamics}
	\label{sec:langevin}
	In this section we discuss the long-time accuracy of the SK-ROCK for Brownian dynamics (also called overdamped Langevin dynamics). We will
	see that using postprocessing techniques we can derive an SK-ROCK method that captures the invariant measure of Brownian dynamics with second order accuracy. In doing so, we do not need our stabilized method to be of weak order 2 on bounded time intervals and we obtain a method that is cheaper than the stochastic orthogonal Runge-Kutta-Chebyshev method of weak order 2 (S-ROCK2) proposed in \cite{AVZ13}, as S-ROCK2 uses many more function evaluations per time-step and a smaller stability domain.
	\subsection{An exact SK-ROCK method for the Orstein-Uhlenbeck process}
	\label{sub:exactOU}
	We consider the $1$-dimensional Orstein-Uhlenbeck problem with $1$-dimensional noise with constants $\delta,\sigma>0$,
\begin{equation} \label{eq:OU1d}
	dX(t) = -\delta X(t)dt  + \sigma dW(t),
	\end{equation}
	that is ergodic and has a Gaussian invariant measure with mean zero and variance given by $\lim_{t\rightarrow\infty}\IE(X(t)^2)=\sigma^2/(2\delta)$.
	Applying the SK-ROCK method to the above system we obtain 
	\begin{equation} \label{eq:SKOU}
	X_{n+1}=A(p)X_n+B(p)\sigma\sqrt{h}\xi_n
	\end{equation}
	where $p=-\delta h$, $\xi_n \sim \calN(0,1)$ is a Gaussian variable and similarly as for  \eqref{eq:defR} we have
	\begin{equation} \label{eq:defROU}
	A(p)=\frac{T_s(\omega_0+\omega_1 p)}{T_s(\omega_0)},\quad B(p)= \frac{U_{s-1}(\omega_0+\omega_1 p)}{U_{s-1}(\omega_0)}(1+ \frac{\omega_1}2 p).
		\end{equation}
		A simple calculation (using that $|A(p)|<1$) gives
	\begin{equation*} 
	\lim_{n\rightarrow\infty}\IE(X_n^2)=\frac{\sigma^2}{2\delta} R(p),\qquad
	R(p)= \frac{2p B(p)^2}{A(p)^2-1}.
	\end{equation*}
	From the above equation, we see that the SK-ROCK method has order $r$ for the invariant measure of \eqref{eq:OU1d} if and only if $R(p)=1+{\cal O}(p^r)$ and a short calculation using \eqref{eq:defROU} reveals that
	$R(p)=1+{\cal O}(p)$, it has order one for the invariant measure (this is of course not surprising because the 
	SK-ROCK  has weak order one).
	We next apply the techniques of postprocessed integrators popular in the deterministic literature \cite{butcher69teo} and proposed in the stochastic context in \cite{Vil15}.
	The idea is to consider a postprocessed dynamics $\overline X_n=G_n(X_n)$ (of negligible cost) such that the process $\overline X_n$ approximates the invariant measure of the dynamical system with higher order. For the process \eqref{eq:OU1d}, we consider the postprocessor
	\begin{equation}
	\label{equ:post_OU}
	\overline X_n =X_n+c\sigma\sqrt{h}\xi_n,
	\end{equation}
	which yields $\lim_{n\rightarrow\infty}\IE(\overline X_n^2)=\frac{\sigma^2}{2\delta} (R(p)-2c^2p)$.
In the case of the SK-ROCK method with $\eta=0$ (zero damping), we have $A(p)=T_s(1+p/s^2),
B(p)=U_{s-1}(1+p/s^2)(1+p/(2s^2))/s$. Setting $c=1/(2s)$ and using the identity
$
(1-x^2)U_{s-1}^2(x)=1-T_{s}^2(x)
$
with $x=1+p/s^2$ reveals that $R(p)-2c^2p=1$ and we obtain
\begin{equation}\label{eq:exact}
\lim_{n\rightarrow\infty}\IE(\overline X_n^2)=\frac{\sigma^2}{2\delta}.
\end{equation}
Hence the postprocessed SK-ROCK method (that will be denoted PSK-ROCK) captures exactly the invariant measure of the $1$-dimensional Orstein-Uhlenbeck problem \eqref{eq:OU1d}. Such a behavior is known for the drift-implicit $\theta$ method with $\theta=1/2$ (see \cite{ChW12} in the context of the stochastic heat equation) and has recently also been shown for the
non-Markovian Euler scheme \cite{LM13}. In \cite{Vil15} an interpretation of the scheme \cite{LM13} as an Euler-Maruyama method with postprocessing \eqref{equ:post_OU} with $c=1/2$
has been proposed and we observe that this is exactly the same postprocessor as for the PSK-ROCK method (with $s=1,\eta=0$). As the SK-ROCK method with zero damping is mean-square stable (see Remark \ref{rem:exactetazero} for $\eta=0$),
it can be seen as a stabilized version of the scheme \cite{LM13}. However, the PSK-ROCK method with $s>1$ and zero damping is not robust to use as its stability domain along the drift axis does not allow for any imaginary perturbation at the points where  $|T_s(1+p/s^2)|=1$ and it is not ergodic (see Remark \ref{rem:etazero} below). 

\paragraph{Stability analysis for Orstein-Uhlenbeck}
	
	Let $M \in\IR^{d\times d}$ denote a symmetric matrix with eigenvalues $-\lambda_d \leq \ldots \leq -\lambda_1 < 0 $, and consider the $d$-dimensional Orstein-Uhlenbeck problem
	\begin{equation} \label{eq:OUmultid}
	dX(t) = MX(t)dt + \sigma dW(t)
	\end{equation}
	where $W(t)$ denotes a $d$-dimensional standard Wiener process.
	The following theorem shows that the damping parameter $\eta>0$ plays an essential role to warranty the convergence to the 
	numerical invariant measure $\rho^h_\infty(x)dx$ at an exponentially fast rate.
	\begin{theorem} \label{thm:stabOU}
		Let $\eta>0$.
		Consider the scheme \eqref{eq:newmethod} with postprocessor \eqref{equ:post_OU} applied to \eqref{eq:OUmultid}
		with stepsize $h$ and stage parameter $s$ such that $2\omega_1^{-1}\geq h\lambda_d$.
		Then, for all $h\leq \eta/\lambda_1$,$\phi\in C_P^1(\IR^d,\IR)$,
		$$|\IE(\phi(\overline X_n)) - \int_{\IR^d}\phi(x)\rho^h_\infty(x)dx | \leq C\exp(-\lambda_1(1+\eta)^{-1} t_n)$$
		where
		$C$ is independent of $h,n,s,\lambda_1,\ldots,\lambda_d$.
	\end{theorem}
	\begin{proof}
		It is sufficient to show the estimate
		\begin{equation} \label{eq:boundR}
		|A(-\lambda_j h)| \leq \exp(-\lambda_1(1+\eta)^{-1} h)
		\end{equation}
		for all $h\leq h_0$,
		where we denote $A(z)={T_s(\omega_0 +\omega_1 z)}/{T_s(\omega_0)}$.
		Indeed, considering two initial conditions $X_0^1,X_0^2$ for \eqref{eq:newmethod} and the corresponding numerical solutions $X_n^1,X_n^2$  (obtained for the same realizations of $\{\xi_n\}$) with postprocessors $\overline X_n^1,\overline X_n^2$, we obtain $X_n^1 - X_n^2 = A(hM) (X_{n-1}^1 - X_{n-1}^2)$ and using the matrix $2$-norm $\|A(hM)\| =\max_j |A(-\lambda_j h)|$ and \eqref{eq:boundR}, we deduce by induction on $n$,
		$$
		\|\overline X_n^1 - \overline X_n^2\| = \|X_n^1 - X_n^2\| \leq \exp(-\lambda_1(1+\eta)^{-1} t_n) \|X_0^1 - X_0^2\|,
		$$
		and taking $\overline X_0^2$ distributed according to the numerical invariant measure yields the result. 
		
		For the proof of \eqref{eq:boundR}, let $z=-\lambda_j h$. 
		Consider first the case $z\in (-\eta\omega_1^{-1}s^{-2},0)$.
		Using the convexity of $A(z)$ on $[-\eta\omega_1^{-1}s^{-2},0]$ (note that $T_s'(x)$ is increasing on $[1,\infty)$), we can bound $A(z)$ by the affine function
		passing by the points $(x_1,A(x_1)),(x_2,A(x_2))$ with 
		$x_1=-\eta\omega_1^{-1}s^{-2}$, $x_2=0$,
		$$
		A(z) \leq 1 + z(1-1/T_s(\omega_0))\eta^{-1}\omega_1s^{2}
		$$
		Using 
		$\omega_1s^{2}\geq 1$ and $T_s(\omega_0)\geq 1+\eta$, we obtain
		$$
(1-1/T_s(\omega_0))\eta^{-1}\omega_1s^{2} \geq 
		(1-(1+\eta)^{-1})\eta^{-1} = (1+\eta)^{-1}.
		$$
		This yields for all $z\in [-\eta\omega_1^{-1}s^{-2},0]$,
		$$
		A(z) \leq 1 + z(1+\eta)^{-1} \leq \exp({z(1+\eta)^{-1}})
		$$
		where we used the convexity of $\exp({z(1+\eta)^{-1}})$ bounded from below by its tangent at $z=0$. We obtain
		$$
		A(-\lambda_j h) \leq e^{-\lambda_j h (1+\eta)^{-1}} \leq e^{-\lambda_1 h (1+\eta)^{-1}}.
		$$
		We now consider the case $z\in
		[-L_s,-\eta \omega_1^{-1}s^{-2}]$. 
		We have $|\omega_0 +\omega_1 z|\leq 1$, thus $|T_s(\omega_0 +\omega_1 z)|\leq 1$ and
		$$|A(z)| \leq 1/T_s(\omega_0) \leq \exp(-\lambda_1(1+\eta)^{-1} h)$$ 
		for all $h\leq (1+\eta)\log (T_s(\omega_0))/\lambda_1$, and thus also for 
		$h\leq \eta/\lambda_1$ where
		we use $T_s(\omega_0) \geq 1+\eta$ and $(1+\eta)\log (1+\eta) \geq \eta$.
		This concludes the proof.
	\end{proof}
	\begin{remark} \label{rem:etazero}
		Note that $\eta>0$ is a crucial assumption in Theorem \ref{thm:stabOU}.
		Indeed, the estimate of Theorem \ref{thm:stabOU} is false for $\eta=0$ already in dimension $d=1$ for all $s>1$: for a stepsize $h$ such that 
		$1-h\lambda_1/s^2=\cos(\pi/s)$ we obtain $A(-\lambda_1h)=-1$ and $B(-\lambda_1h)=0$ in \eqref{eq:defROU}  (corresponding to the local extrema $p=-\lambda_1 h$ of $A(p)$ closest to zero)
		and $X_{n}=(-1)^n X_{0}$ for all $n$, and the scheme is not ergodic.
		In addition, notice that Theorem \ref{thm:stabOU} allows to use an $h$-dependent value 
		of $\eta$ such as $\eta= h\widetilde \lambda_1$
		where $\widetilde \lambda_1\geq \lambda_1$ is an upper bound for $\lambda_1$.
		In this case, the exponential convergence of Theorem \ref{thm:stabOU} holds for all stepsize
		$h\leq 1$.
	\end{remark}
	
We end this section by noting that being exact for the invariant measure of Brownian dynamics \eqref{eq:langevin} is only true for the PSK-ROCK method (or the method in \cite{LM13}) 
in the linear case, i.e. for a quadratic potential $V$.
Second order accuracy for the invariant measure 
has been shown in \cite{LMT14} (see also \cite{Vil15}) for the method in \eqref{eq:LM} (equivalent to PSK-ROCK with $s=1$ stage) for general
nonlinear Brownian dynamics \eqref{eq:langevin}. This will also be shown for the nonlinear PSK-ROCK method in the next section.

	\subsection{PSK-ROCK: a second order postprocessed SK-ROCK method for nonlinear Brownian dynamics}
	We consider the overdamped Langevin equation,
	\begin{equation} \label{eq:langevin}
	dX(t) = -\nabla V(X(t))dt + \sigma dW(t),
	\end{equation}
	where the stochastic process $X(t)$ takes values in $\IR^d$ and $W(t)$ is a $d$-dimensional Wiener process.
	We assume that the potential $V:\IR^d \to \IR$ has class $C^\infty$ and satisfies the at least quadratic growth assumption
	\begin{equation}\label{eq:Vgrowth}
	x^T\nabla V(x) \geq C_1x^T x -C_2
	\end{equation}
	for two constants $C_1,C_2>0$ independent of $x\in\IR^d$.
	The above assumptions warranty that the system \eqref{eq:langevin} is ergodic with exponential convergence
	to a unique invariant measure with Gibbs density $\rho_\infty = Z \exp(-2\sigma^{-2}V(x))$,
	$$
	|\IE(\phi(X(t)) - \int_{\IR^d} \phi(x) \rho_\infty(x) dx| \leq C e^{-\lambda t},
	$$
	for test function $\phi$ and all initial condition $X_0$, where $C,\lambda$ are independent of $t$.
	
	We propose to modify the internal stage $K_1=X_0+\mu_1hf(X_0+\nu_1Q)+\kappa_1Q$ of the method \eqref{eq:newmethod} as follows:
	\begin{equation} \label{eq:defR1new}
	K_1=X_0+\mu_1hf(X_0+\nu_1Q)+\kappa_1Q+\alpha h\big(f(X_0+\nu_1Q)-2 f(X_0)+f(X_0-\nu_1Q)\big),
	\end{equation}
	where $\alpha$ is a parameter depending on $s$ and $\eta$ given in 
Theorem \ref{thm:postprocessor} below.

	\begin{remark}
		Notice that for $\alpha=0$, we recover the original definition from \eqref{eq:newmethod}. 
		We note that the parameter $\alpha$ does not modify 
		the stability function of Lemma \ref{lemma:stabfunction}, and yields a perturbation of order $\bigo(h^2)$ in the definition of $X_1$.
		Thus, the results of Theorem \ref{thm:stabsize} and Theorem \ref{thm:conv}
		remain valid for any value of $\alpha$ for the scheme \eqref{eq:newmethod}
		with modified internal stage \eqref{eq:defR1new}.
	\end{remark}

	\begin{theorem} \label{thm:postprocessor}
	Consider the Brownian dynamics \eqref{eq:langevin}, where we assume that $V:\IR^d \to \IR$ has class $C^\infty$, with $\nabla V$ globally Lipschitz and satisfying \eqref{eq:Vgrowth}.
		Consider the scheme \eqref{eq:newmethod} applied to \eqref{eq:langevin} with modified internal stage $K_1$ defined in \eqref{eq:defR1new} with $\alpha$ defined in \eqref{eq:alphac},
	and the postprocessor defined as
		$$
		\overline X_n = X_n + c\sigma \sqrt h \xi,
		$$
		where 
				\begin{equation} \label{eq:alphac}
				c^2 = -\frac14 + \frac{\omega_1}2 +\frac{\omega_1T_s''(\omega_0)}{T_s'(\omega_0)}-\frac{\omega_1^2T_s''(\omega_0)}{4T_s(\omega_0)}, \qquad
		\alpha =\frac{2}{s\omega_0\omega_1} (c^2+ \frac{\omega_1^2T_s''(\omega_0)}{2 T_s(\omega_0)}-r_s),
		\end{equation}
		and  $r_s$ is defined by induction as $r_0=0$, $r_1=\frac{s^2\omega_1^3}{4\omega_0}:=\Delta_1$
    and
		$$
		r_i=\nu_ir_{i-1}+\kappa_ir_{i-2}+\Delta_i,\qquad \Delta_i=\mu_i\frac{sT_{i-1}'(\omega_0)\omega_1}{(i-1)T_{i-1}(\omega_0)},\qquad i=2,\ldots s.
		$$
		Then, $\overline X_n$ yields order two for the invariant measure, i.e. \eqref{eq:difference1i} holds with $r=2$, and in addition
		\begin{equation} \label{eq:convtwo}
		|\IE(\phi(\overline X_n) - \int_{\IR^d} \phi(x) \rho_\infty(x) dx| \leq C_1 e^{-\lambda t_n} + C_2h^2,
		\end{equation}
		for all $t_n=nh$, $\phi\in C_P^\infty(\IR^d,\IR)$, where $C_1,C_2$ are independent of $h$ assumed small enough, and $C_2$ is independent of the initial condition $X_0$.
	\end{theorem}
		The proof of Theorem \ref{thm:postprocessor} relies on the following postprocessing analysis  from \cite{Vil15}.
	Consider a scheme \eqref{eq:methnum} with bounded moments and assumed ergodic when applied to \eqref{eq:langevin}, where the potential $V$ satisfies the above ergodicity assumptions. 
	Assume that the scheme has a weak Taylor expansion after one time step of the form
	\begin{equation}\label{eq:weaktaylor}
	\IE(\phi(X_1)|X_0=x) = \phi(x) + h\calL \phi(x) + h^2\mathcal{A}_1 \phi(x) + \bigo(h^3),
	\end{equation}
		and consider a postprocessor of the form $\overline X_n = G_n(X_n)$ where 
			\begin{equation}\label{eq:weaktaylorproc}
	\IE(\phi(\overline X_1)|X_1=x) = \phi(x) + h\overline{\mathcal{A}}_1 \phi(x) + \bigo(h^3),
	\end{equation}
	where the constants in $\bigo$ in \eqref{eq:weaktaylor},\eqref{eq:weaktaylorproc} have at most a polynomial growth with respect to $x$. Here $\mathcal{L}\phi=\phi'f+\sigma^2/2 \Delta \phi$ denotes generator of the SDE and $\mathcal{A}_1,\overline{\mathcal{A}}_1$ are linear differential operators with smooth coefficients. 
	Note that $\mathcal{A}_1\ne \mathcal{L}^2/2$ in general (otherwise the scheme has weak order 2).
	If the  condition $(A_1 + [\calL,\overline{\mathcal{A}}_1])^* \rho_\infty =0$ holds,
	equivalently,
		\begin{equation}\label{eq:condproc}
		\left\langle \mathcal{A}_1\phi +[\calL,\overline{\mathcal{A}}_1] \phi \right\rangle = 0
		\end{equation}
		for all test function $\phi$, 
		where we define $\left\langle \phi \right\rangle = \int_{\mathbb{R}^d} \phi \rho_\infty dx$,
		then it is shown in \cite[Theorem 4.1]{Vil15} that 
		$\overline X_n$ has order two for the invariant measure, i.e. the convergence estimates \eqref{eq:difference1i} with $r=2$ and \eqref{eq:convtwo} hold.
	Before we can apply the above result, the following lemma 
	allow to compute the weak Taylor expansion of the modified scheme.
	\begin{lemma} \label{lemma:postprocessor}
	Consider the scheme \eqref{eq:newmethod} with modified stage \eqref{eq:defR1new}
	and assume the hypotheses of Theorem \ref{thm:postprocessor}.
	Then \eqref{eq:weaktaylor} holds where
	the linear differential operator $\mathcal{A}_1$ is given by
		\begin{eqnarray} \label{eq:Alangevin}
		\mathcal{A}_1\phi &=& \frac12 \phi''(f,f)  + \frac{\sigma^2}2 \sum_{i=1}^d\phi'''(e_i,e_i,f) + \frac{\sigma^4}{8} \sum_{i,j=1}^d \phi^{(4)}(e_ie_i,e_j,e_j)
		+ c_2 \phi'f'f \nonumber\\
		&+&c_3 \frac{\sigma^2}2 \phi' \sum_{i=1}^df''(e_i,e_i) + c_4 \sigma^2 \sum_{i=1}^d \phi''(f'e_i,e_i),
		\end{eqnarray}
		where $f=-\nabla V(x)$ and
		\begin{equation}\label{eq:c234}
		c_2 = \frac{\omega_1^2T_s''(\omega_0)}{2 T_s(\omega_0)},\qquad
		c_3 = r_s+\frac{\omega_0}{s\omega_1}\alpha,\qquad
		c_4 = \frac{T_s''(\omega_0)\omega_1}{T_s'(\omega_0)} + \frac{\omega_1}2.
		\end{equation}
	\end{lemma}
	\begin{proof}
	Adapting the proof of Lemma \ref{lemma:localerror}, the internal stage $K_i$ defined in \eqref{eq:newmethod} (and \eqref{eq:defR1new} for $i=1$)
		satisfies \eqref{eq:K_i} where $h^2R_{i,h}(X_0)$ can be replaced by
		\begin{equation} \label{eq:K_inew}
		\frac{\omega_1^2T_i''(\omega_0)}{2 T_i(\omega_0)} h^2f'(X_0)f(X_0) 
		+ \tilde{r}_i \frac{\sigma^2}2 f''(X_0)(\xi_n,\xi_n) + h^{5/2} \widetilde R_{i} +  h^3 \widetilde R_{i,h}(X_0),
		\end{equation}
		where $\IE(\widetilde R_{i})=0$ and all the moments of $\widetilde R_{i},\widetilde R_{i,h}(X_0)$ are bounded with polynomial growth with respect to $X_0$. Here, $\tilde{r}_i$ is defined by induction as $\tilde{r}_0=0$, 
		$\tilde{r}_1=\Delta_1+\alpha$, and 
		$$
		\tilde{r}_i=\nu_i\tilde{r}_{i-1}+\kappa_i\tilde{r}_{i-2}+\Delta_i,\quad i=2,\ldots,s.
		$$
			We have $\IE(\widetilde R_{i})=0$ because $\widetilde R_{i}$ is a linear combination of
			$f'(X_0)f'(X_0)\xi_n$, $f''(X_0)(f(X_0),\xi_n)$, and $f'''(X_0)(\xi_n,\xi_n,\xi_n)$ with zero mean values (recall that odd moments of $\xi_n$ vanish).
			Next, observing that the difference $d_i=\tilde{r}_i-r_i$ satisfies $d_0=0,d_1=\alpha$,
			and 
			$
			d_i=\nu_id_{i-1}+\kappa_id_{i-2},i=2,\ldots,s,
			$
			we deduce $$\tilde{r}_i=r_i+d_i,\qquad d_i=\frac{U_{i-1}(\omega_0)}{T_i(\omega_0)}\omega_0\alpha~~\forall~~i=0,..,s.$$
			In particular, taking $i=s$ in \eqref{eq:K_i},\eqref{eq:K_inew},
			and expanding \eqref{eq:weaktaylor}, we deduce that \eqref{eq:Alangevin} holds
			with $c_2,c_3,c_4$ defined in \eqref{eq:c234} where we note that $c_3=\widetilde r_s=r_s+d_s$.
	\end{proof}
	
	\begin{proof}[Proof of Theorem \ref{thm:postprocessor}]
	Following the proof of \cite[Theorem 4.2]{Vil15} 
	(see also \cite[Theorem 5.8]{LaV17})
	where we apply repeatedly integration by parts for the integral in \eqref{eq:condproc}, 
	using Lemma \ref{lemma:postprocessor} for the expression of 
	$\mathcal{A}_1$, we deduce that the quantity in \eqref{eq:condproc}
	satisfies
		$$
		\left\langle \mathcal{A}_1\phi +[\calL,\overline{\mathcal{A}}_1]\phi \right\rangle = \sum_{i=1}^d\left\langle  (c_3-c_2-c^2) \frac{\sigma^2}2 \phi' f''(e_i,e_i) + (c_4-\frac14 - \frac{c_2}2 -c^2) \sigma^2 \phi''(f'e_i,e_i) \right\rangle,
		$$
		where we use 
		$[\calL,\overline{\mathcal{A}}_1]=
	-c^2\sigma^2\big(1/2\, \phi' \sum_{i=1}^df''(e_i,e_i) +  \sum_{i=1}^d \phi''(f'e_i,e_i) \big)$ for $\overline{\mathcal{A}}_1\phi =c^2\sigma^2/2\, \Delta \phi$.
		We see that the above quantity \eqref{eq:condproc} vanished if $c_3-c_2-c^2=c_4-\frac14 - \frac{c_2}2 -c^2=0$, equivalently,
		\begin{equation} \label{eq:ordercond2}
		c_3-c_2  = c^2 = c_4-\frac14 - \frac{c_2}2.
		\end{equation}
		For the values of $\alpha,c$ defined in \eqref{eq:alphac}, we obtain that \eqref{eq:ordercond2} indeed holds and we deduce that the order two condition \eqref{eq:condproc} for the invariant measure is satisfied. This concludes the proof.
	\end{proof}
	%
	%
	%
		
\section{Numerical experiments}
	\label{sec:numerical}
	
	In this Section, we illustrate numerically our theoretical analysis and we show the performance of the proposed SK-ROCK method and its postprocessed modification PSK-ROCK.
	
	\subsection{A nonlinear nonstiff problem}
	\begin{figure}[tb]
		\centering
		\begin{subfigure}[t]{0.45\textwidth}
			\centering
			\includegraphics[width=1\linewidth]{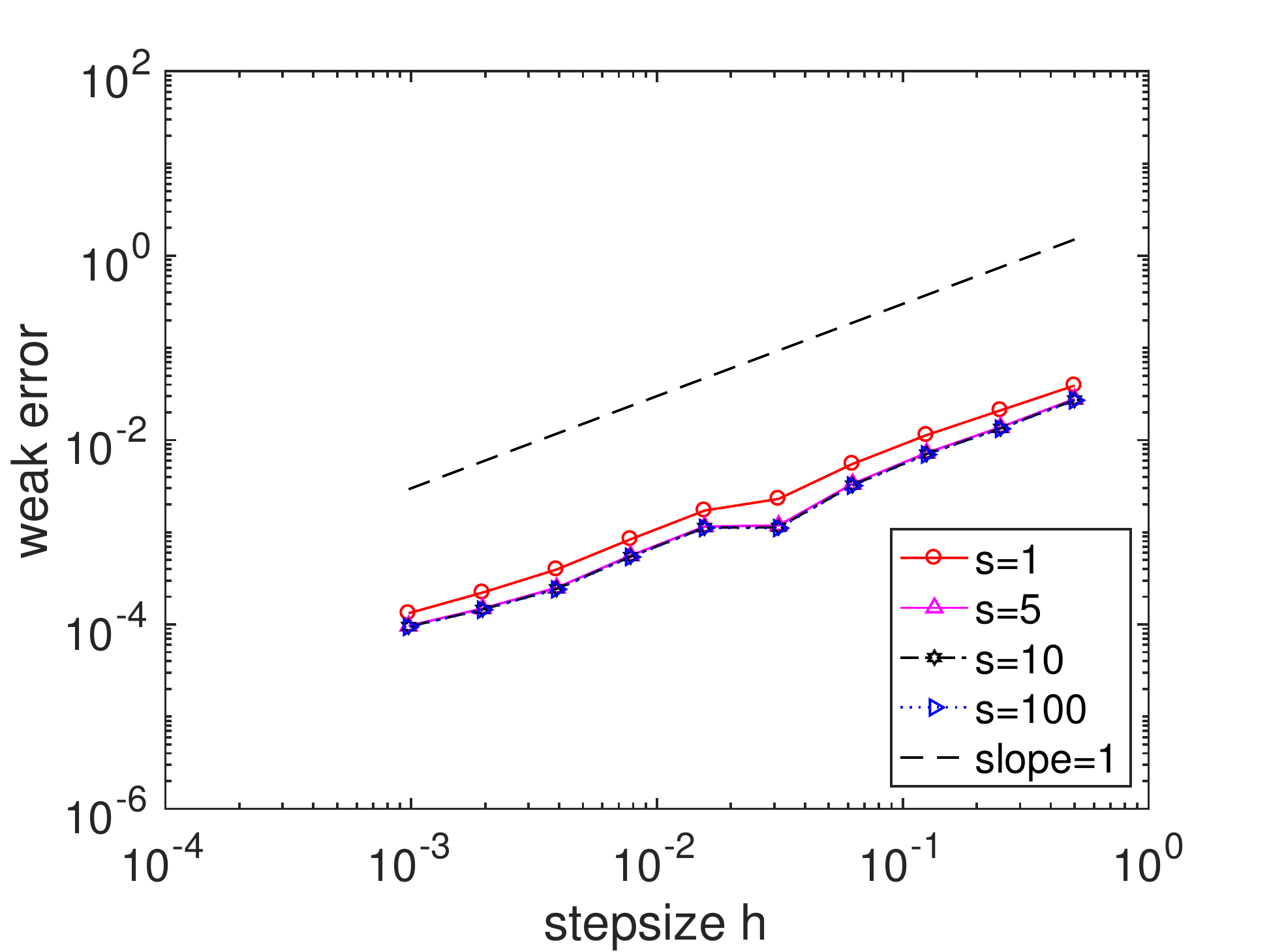}
			\caption{Weak error\\ $|\IE(\arcsinh(X(T))-\IE(\arcsinh(X_N))|.$}
		\end{subfigure}
		\begin{subfigure}[t]{0.45\textwidth}
			\centering
			\includegraphics[width=1\linewidth]{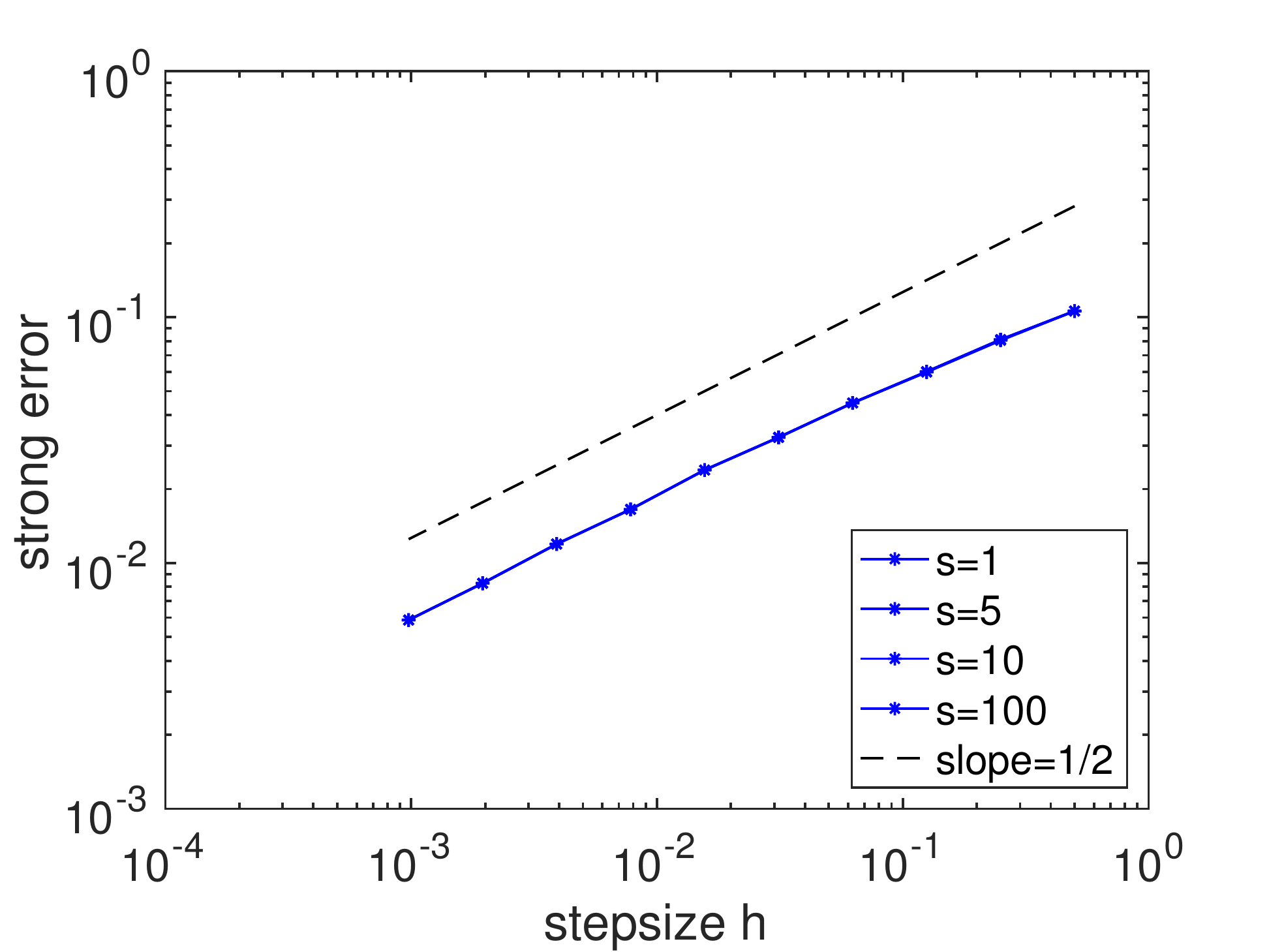}
			\caption{Strong error $\IE(|X(T)-X_N|$).}
		\end{subfigure}
		\caption{Nonlinear problem \eqref{eq:pb1}. Strong and weak convergence plots using SK-ROCK with final time $T=1$, stepsizes $h=2^{-p},p=1..10$, $10^4$ samples and number of stages $s=1,5,10,100.$}
		\label{fig:pb1}
	\end{figure}
	We first consider the following non-stiff nonlinear SDE,
	\begin{equation}\label{eq:pb1}
	dX=\left(\frac14X+\frac12\sqrt{X^2+1}\right)dt+\sqrt{\frac12(X^2+1)}dW,\quad X(0)=0.
	\end{equation}
	whose exact solution is	$X(t)=\sinh(\frac t2+\frac{W(t)}{\sqrt2})$.
	In Figure \ref{fig:pb1}, we consider the SK-ROCK method \eqref{eq:newmethod} and
	plot the strong error $\IE(|X(T)-X_N|$ and the weak error $|\IE(\arcsinh(X(T))-\IE(\arcsinh(X_N))|$
	at the final time $T=Nh=1$ using $10^4$ samples and number of stages $s=1,5,10,100$.
	We obtain convergence slopes 1 and 1/2, respectively, which confirms Theorem \ref{thm:conv} stating the strong order $1/2$ and weak order $1$ of the proposed scheme. 
	Note that $s=1$ stage is sufficient for the stability of the scheme in the non-stiff case. The results for $s=5,10,100$ yield nearly identical curves which illustrates that the error constants of the method are nearly independent of the stage number of the scheme.

	\subsection{Nonlinear nonglobally Lipschitz stiff problems}
	\begin{figure}[tb]
		\centering
		\begin{subfigure}[t]{0.45\textwidth}
			\centering
			\includegraphics[width=1\linewidth]{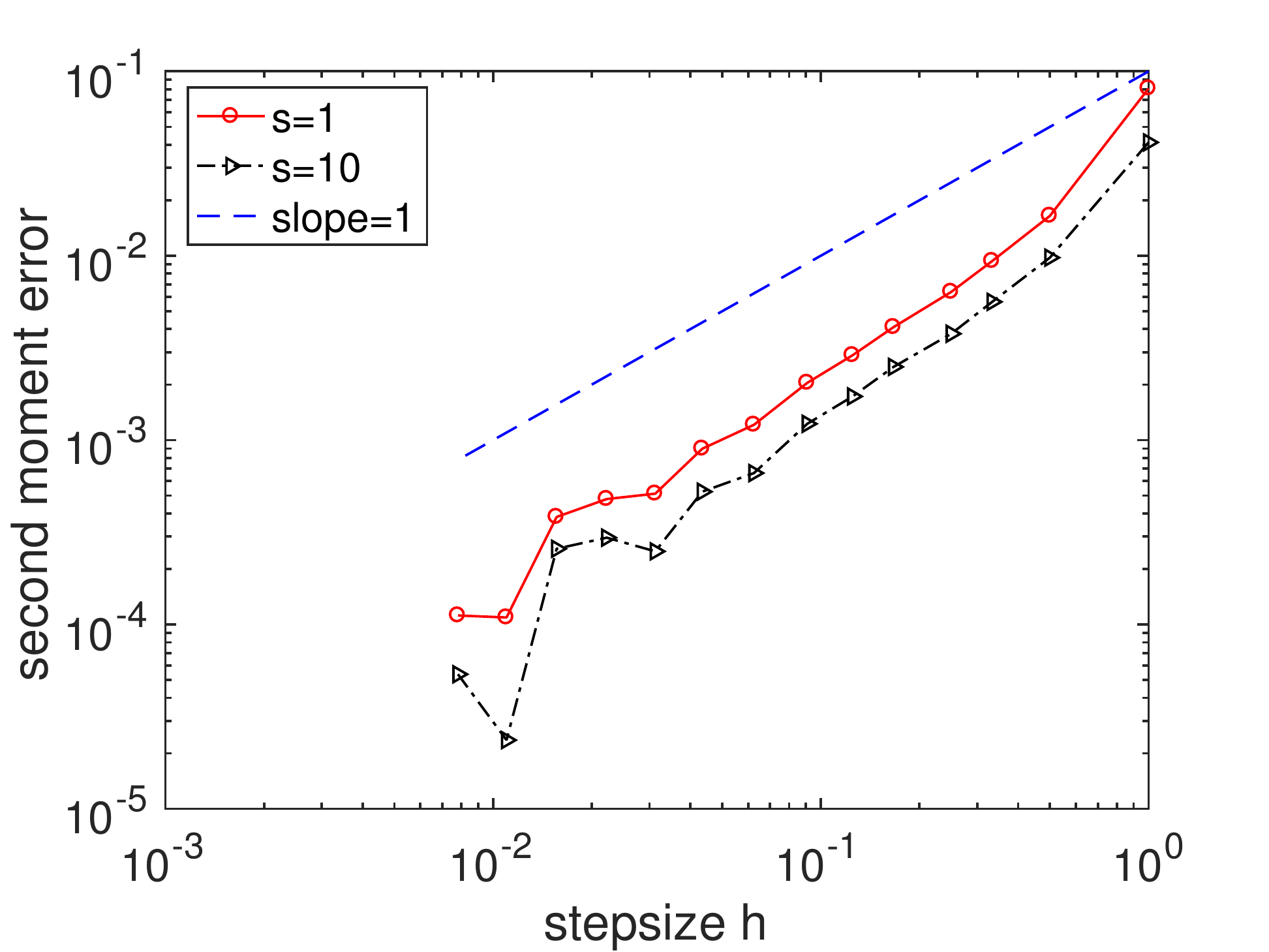}
			\caption{Non-stiff case $-\lambda_1=\mu_1=1$.}
		\end{subfigure}
		\begin{subfigure}[t]{0.45\textwidth}
			\centering
			\includegraphics[width=1\linewidth]{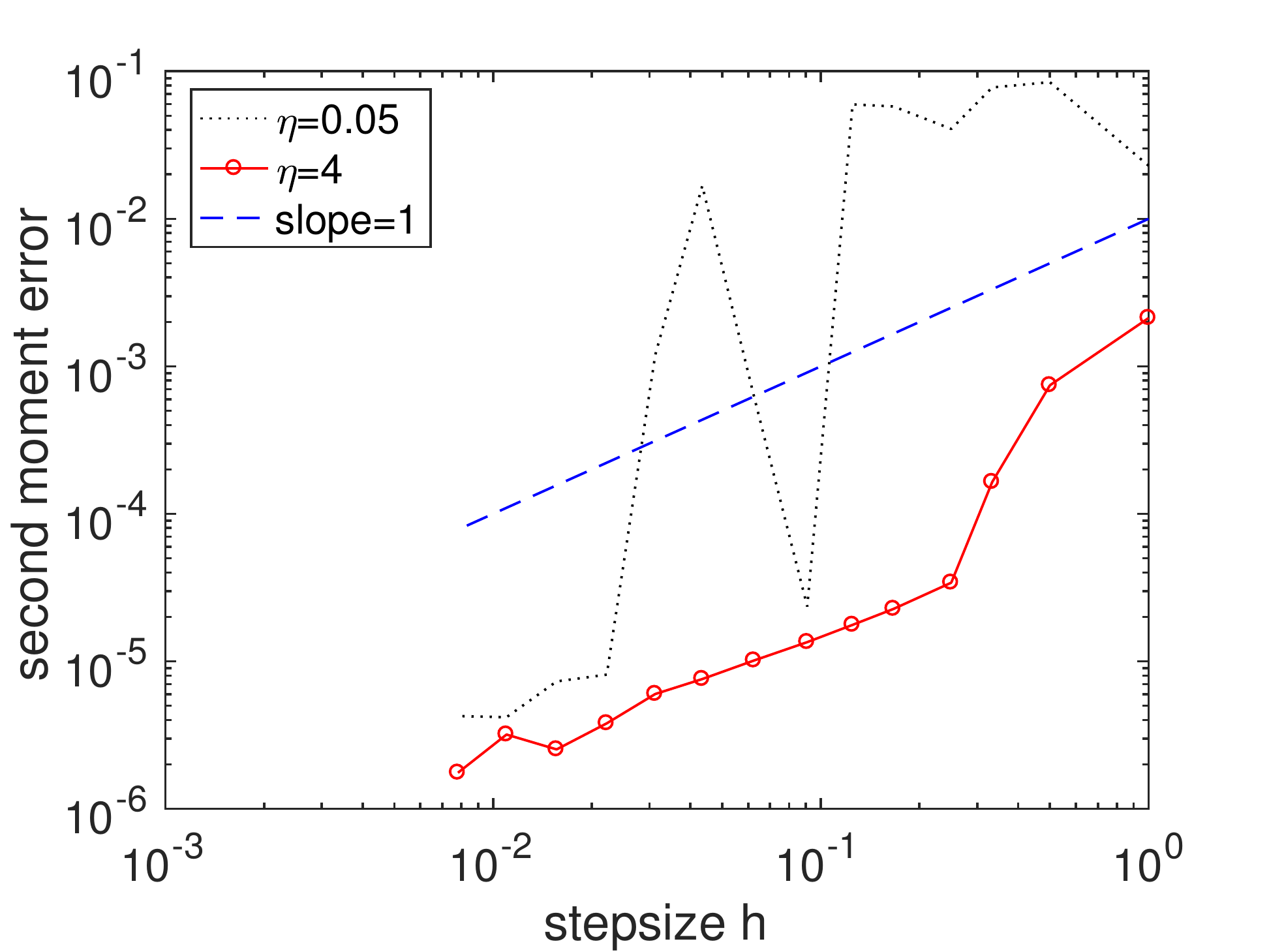}
			\caption{Stiff case $-\lambda_1=\mu_1^2=100$. } \label{fig:populationb}
		\end{subfigure}
		\caption{Nonlinear problem \eqref{eq:population} with $\nu=2,\mu_2=0.5$,$\lambda_2=-1$. Weak convergence plots using SK-ROCK for $\IE(X(T)^2)$ where $T=1$, $h=T/[2^{i/2}],i=1,\ldots,14$, and $10^6$ samples. For the stiff case 
		(b), the method uses the following number of stages respectively: $s=8,6,5,4,4,3,3,3,2,2,2,1,1,1$ (with damping $\eta=0.05$) and $s=13,9,8,7,6,5,4,4,3,3,3,3,2,2$ (with damping $\eta=4$).} 
		\label{fig:population}
	\end{figure}
	
	Consider the following nonlinear SDE in dimensions $d=2$ with a one-dimensional noise ($d=2,m=1$). This is a modification of a one-dimensional population dynamics model from \cite[Chap.\ts 6.2]{Gar88} considered in \cite{AVZ12b,AVZ13,AbB2} for testing stiff integrator performances,
	\begin{align}\label{eq:population}
	dX&=(\nu(Y-1)-\lambda_1 X(1-X))dt-\mu_1X(1-X)dW,\quad X(0)=0.95,\nonumber \\
	dY&=-\lambda_2Y(1-Y)dt-\mu_2Y(1-Y)dW,\quad Y(0)=0.95.
	\end{align}
	Observe that linearizing \eqref{eq:population} close to the equilibrium $(X,Y)=(1,1)$, we recover for $\nu=0$ the scalar test problem \eqref{eq:lintest}. In Figure \ref{fig:population}
	we consider the SK-ROCK method applied to \eqref{eq:population} with parameters that are identical to those used in \cite[Sect.\ts4.2]{AbB2}. We take
	the initial condition $X(0)=Y(0)=0.95$ close to this steady state and use the parameters $\nu=2,\mu_2=0.5$,$\lambda_2=-1$.
	In a nonstiff regime ($-\lambda_1=\mu_1=1$ in Figure \ref{fig:population}(a)), we observe a convergence slope $1$ for the second moment $\IE(X(T)^2)$
which illustrates the weak order one of the scheme, although our analysis in Theorem \ref{thm:conv} applies only for globally Lipschitz vector fields. The stage number $s=1$ is sufficient for stability, but we also include for comparison the results for $s=10$ (note that the results for $s=50,100$ not displayed here are nearly identical to the case $s=10$). The convergence curves are obtained as an average over $10^6$ samples. 
	In a stiff regime ($-\lambda_1=\mu_1^{2} =100$ in Figure \ref{fig:population}(b)), we observe for the standard small damping $\eta=0.05$ a stable but not very accurate convergence, due to the severe nonlinear stiffness.
However, considering a slightly larger damping $\eta=4$, in the spirit of the S-ROCK method, yields a stable integration for all considered timesteps and all trajectories and we observe a line with slope one for the SK-ROCK method. Here, given the timesteps~$h$, the numbers of stages $s$ are adjusted as proposed in \eqref{eq:defseta} where $\lambda_{\max}=|\lambda_1|=100$.

	\begin{remark} \label{rem:stabilization}
		For severely stiff problems, 
		alternatively to switching to drift-implicit schemes \cite{Hig00,AVZ12b},
		one can consider in SK-ROCK a slightly larger damping $\eta$ and the corresponding stage parameter $s$ below, similar to \eqref{eq:defsnum} and chosen such that the mean-square stability domain length \eqref{eq:defL} satisfies $L > h\lambda_{\max}$,
		\begin{equation} \label{eq:defseta}
		s=\left[\sqrt{\frac{h\lambda_{\max}+1.5}{2\Omega(\eta)}}+0.5\right],
		\end{equation} 
		where $\Omega(\eta)$ is given in Lemma \ref{limits}.
	\end{remark}
	
	\begin{figure}[tb]
		\centering
		\begin{subfigure}[t]{0.45\textwidth}
			\centering
			\includegraphics[width=1\linewidth]{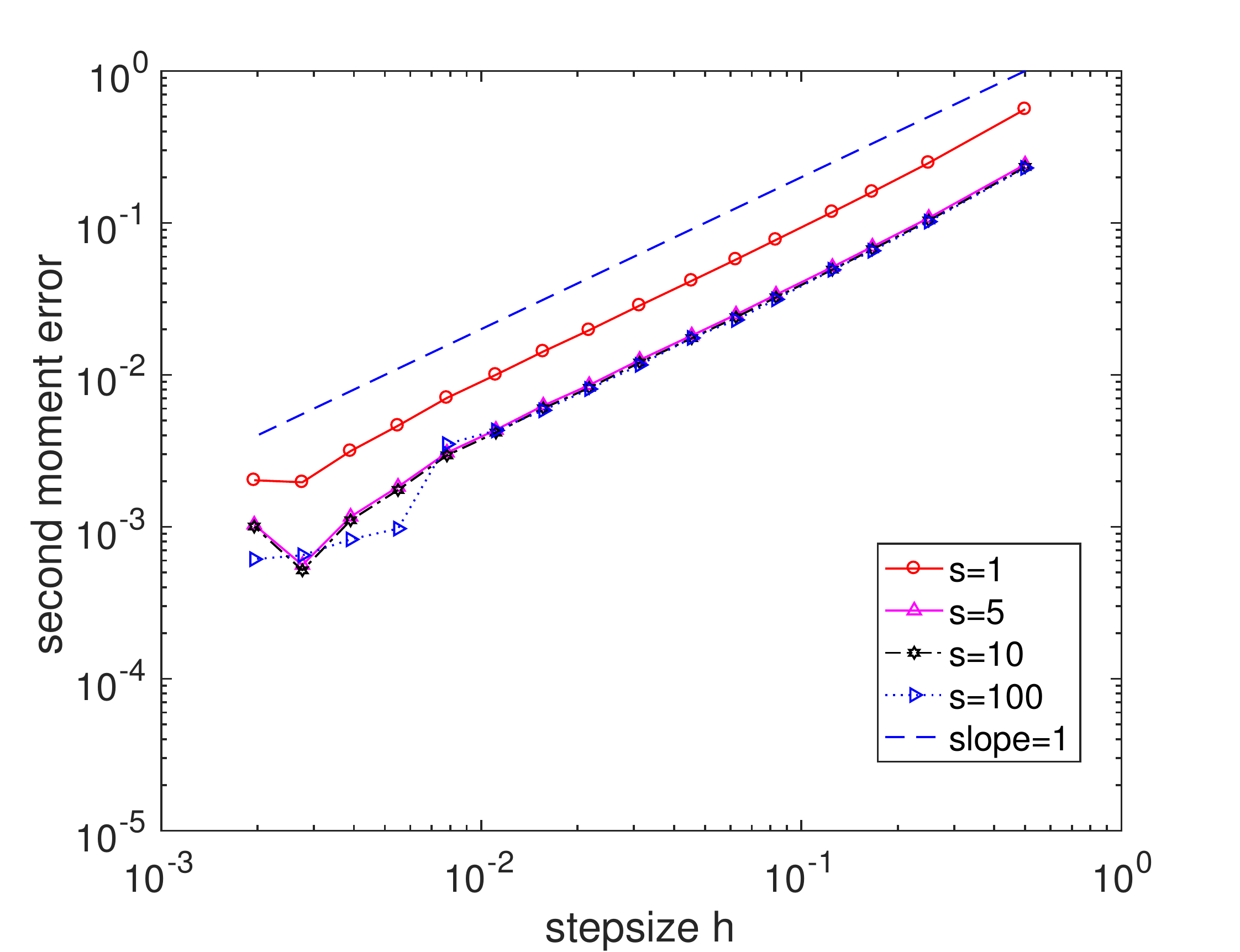}
			\caption{SK-ROCK with $T=0.5$.}
		\end{subfigure}
		\begin{subfigure}[t]{0.45\textwidth}
			\centering
			\includegraphics[width=1\linewidth]{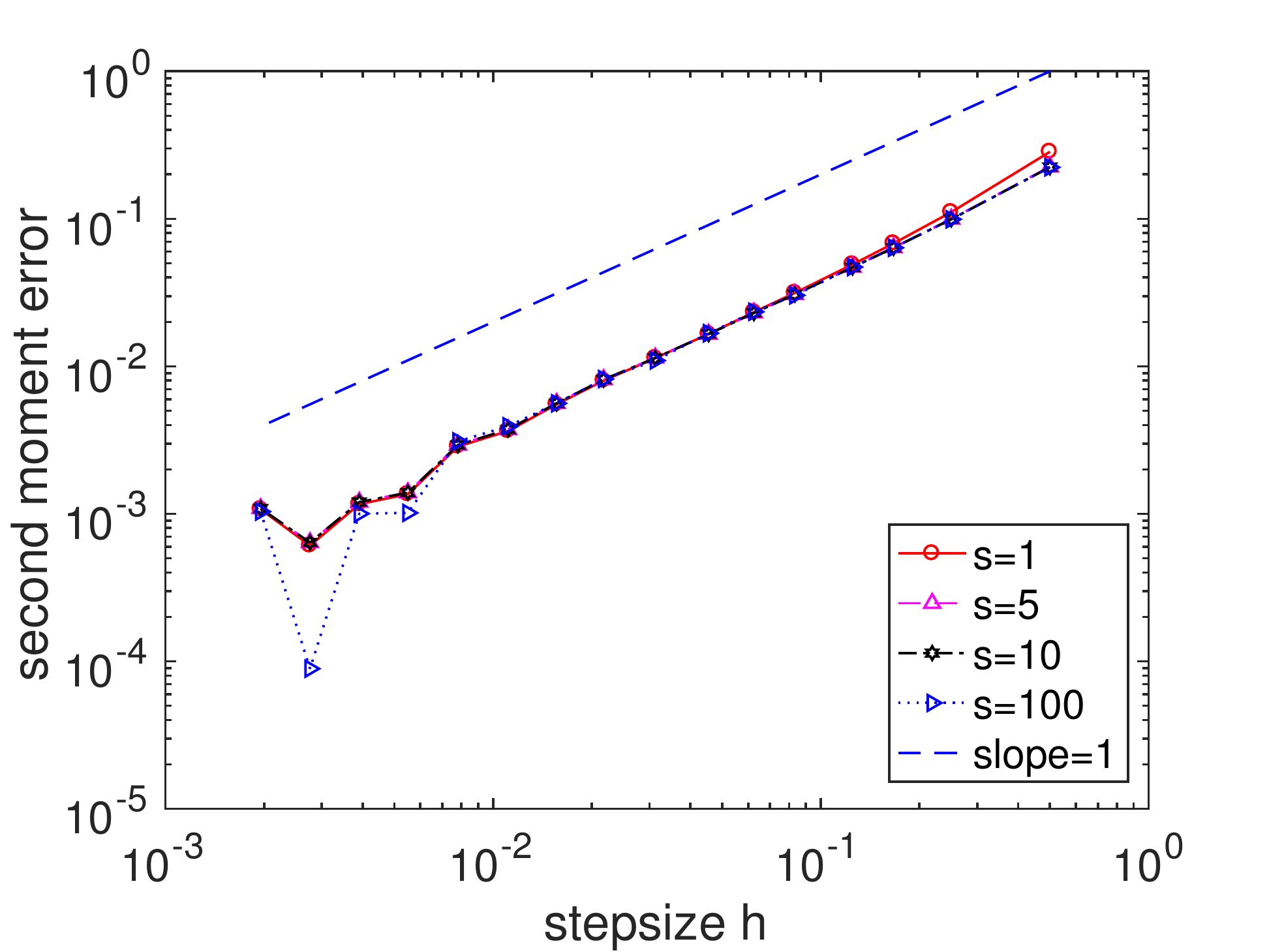}
			\caption{PSK-ROCK with $T=0.5$.}
		\end{subfigure}\\
		\begin{subfigure}[t]{0.45\textwidth}
			\centering
			\includegraphics[width=1\linewidth]{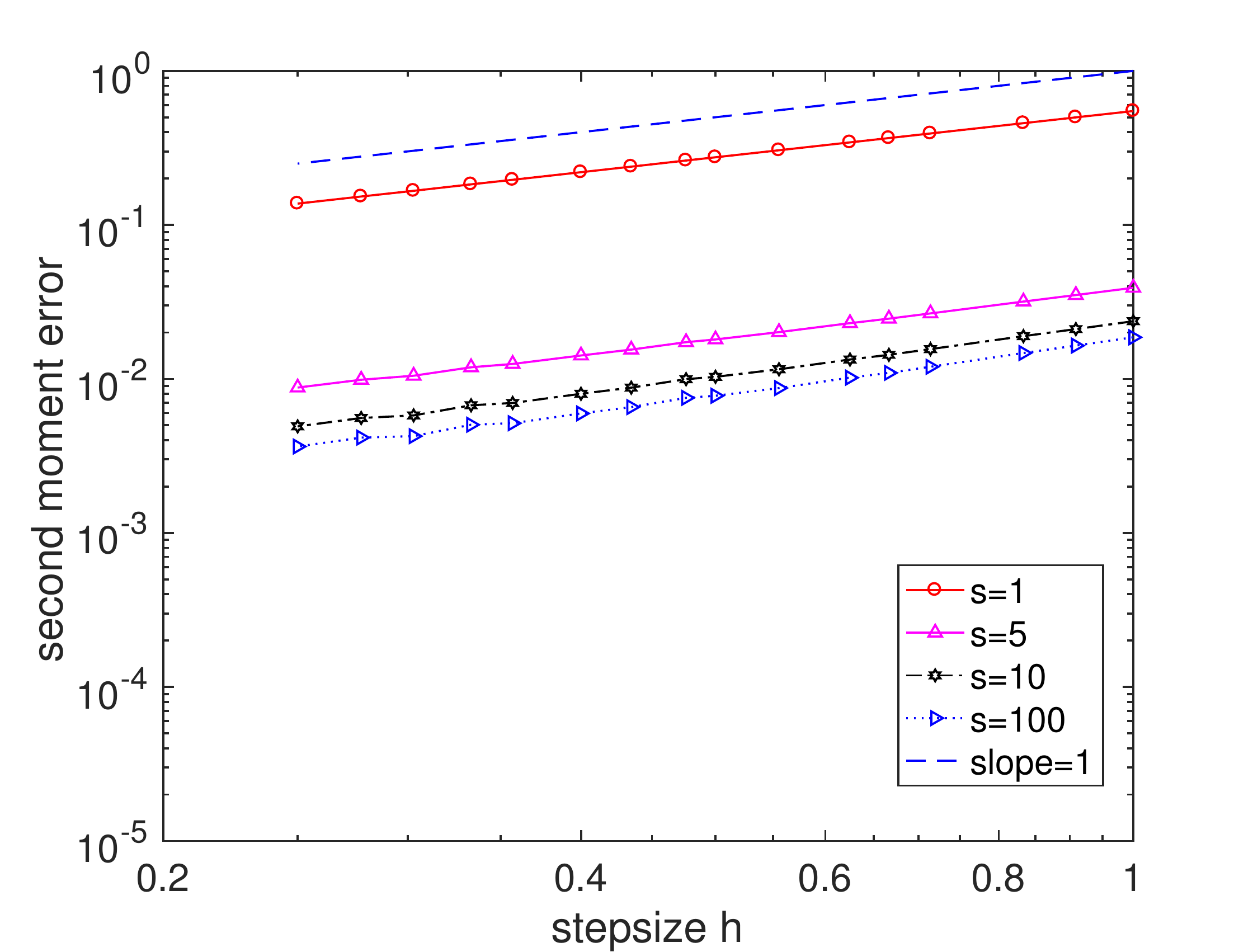}
			\caption{SK-ROCK with $T=10$.}
		\end{subfigure}
		\begin{subfigure}[t]{0.45\textwidth}
			\centering
			\includegraphics[width=1\linewidth]{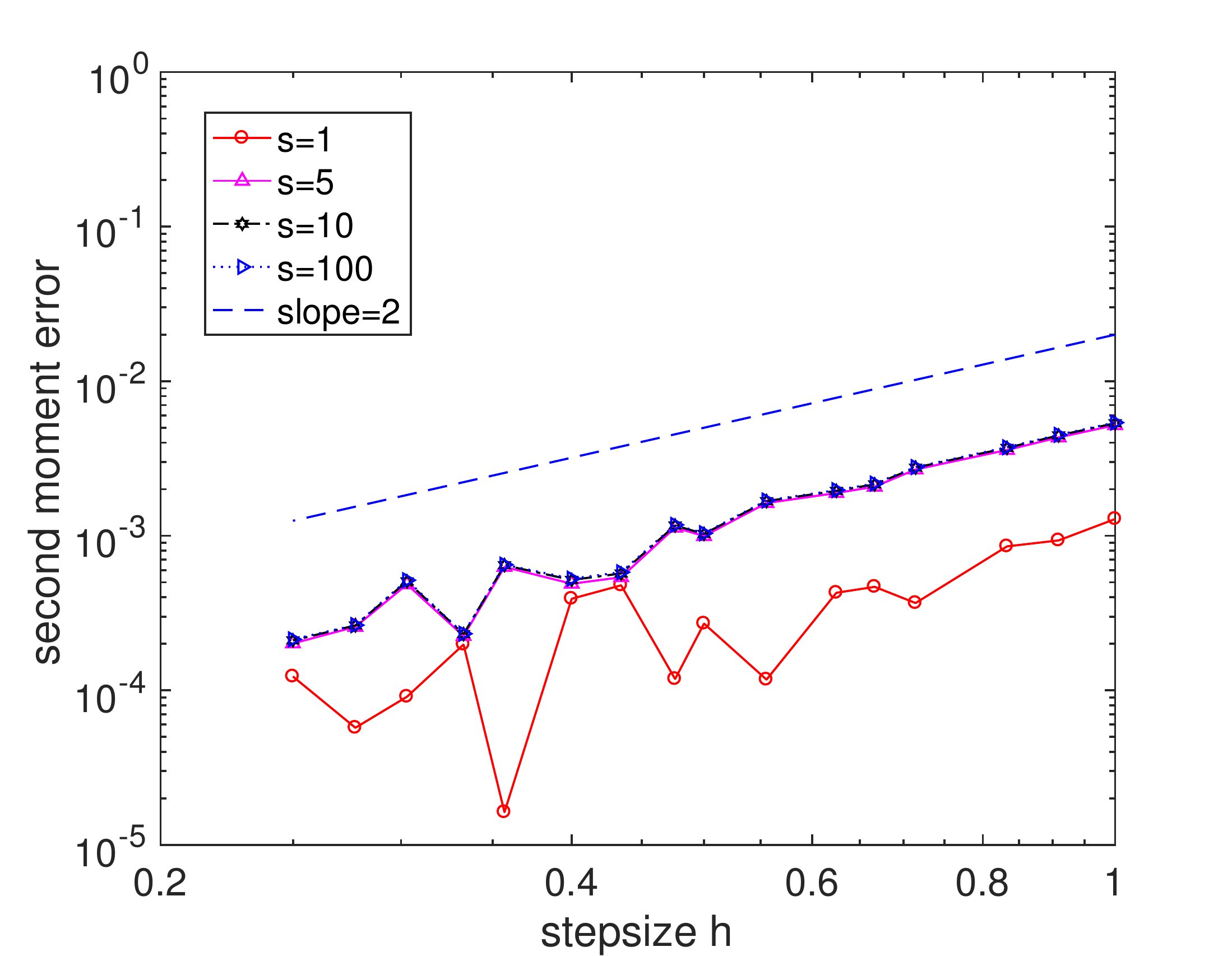}
			\caption{PSK-ROCK with $T=10$.}
		\end{subfigure}
		\caption{Linear additive problem \eqref{OU}. Second moment error $\IE(X(T)^2)$ for for short time $T=0.5$ (top pictures) and long time $T=10$ (bottom pictures) without (SK-ROCK) or with a postprocessor (PSK-ROCK).
			where, $h=T/[10\times2^{i/8}]$ for $T=10$, and $h=T/[2^{i/2}]$ for $T=0.5$ with $i=1,\ldots,16,$ 
			and $10^8$ samples.}
		\label{fig:OU}
	\end{figure}
	
	\subsection{Linear case: Orstein-Uhlenbeck process}
	We now illustrate numerically in details the role of the postprocessor introduced in Theorem \ref{thm:postprocessor}
	for the linear Orstein-Uhlenbeck process in dimension $d=m=1$,
	\begin{equation}\label{OU}
	dX=-\lambda Xdt+\sigma dW,\quad
	X(0)=2
	\end{equation}
	where we choose $\lambda=1$ and $\sigma=\sqrt2$.
	
In Figure \ref{fig:OU}, we consider the SK-ROCK and PSK-ROCK methods with $s=1,5,10,100$ stages, respectively.
	For a short time $T=0.5$ (Fig. \ts\ref{fig:OU}(a)(b)), we observe weak convergence slopes one for both SK-ROCK and PSK-ROCK (second moment $\IE(X(T)^2)$)  as predicted by Theorem \ref{thm:conv}, and the postprocessor has nearly no effect of the errors.
	For a long time $T=10$ where the solution of this ergodic SDE is close to equilibrium, we observe that the weak order one of SK-ROCK (Fig. \ts\ref{fig:OU}(c)) is improved to order two using the postprocessor in PSK-ROCK (Fig. \ts\ref{fig:OU}(d)), which confirms the statement of Theorem \ref{thm:postprocessor} that the postprocessed scheme has order two of accuracy for the invariant measure.
	\begin{figure}[tb]
		\centering
		\includegraphics[scale=0.35]{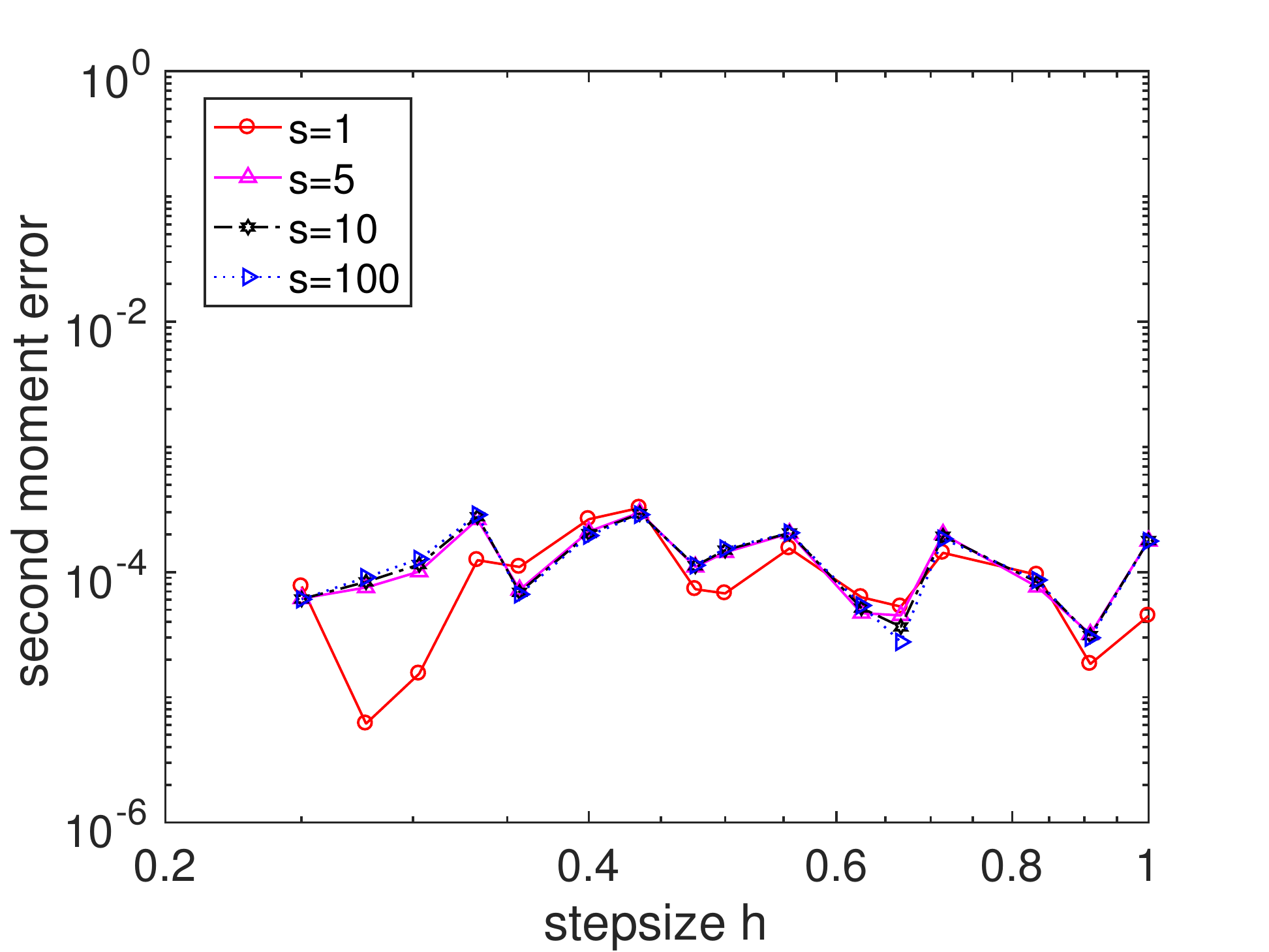}
		\caption{PSK-ROCK without damping $(\eta=0)$.
			Second moment error of problem \eqref{OU}, with $T=10$, $h=T/[2^{i/2}]$, and $s=1,5,10,100$, using $M=10^8$ samples (the Monte-Carlo error has size $M^{-1/2}=10^{-4}$).}
		\label{fig:OUexact}
	\end{figure}
	For comparison, in Figure \ref{fig:OUexact}, we also include the results of PSK-ROCK without damping ($\eta=0$) using $M=10^8$ samples. We recall that for the scalar linear  Orstein-Uhlenbeck process, the PSK-ROCK method with zero damping is exact for the invariant measure (see Section \ref{sub:exactOU}).
	We observe only Monte-Carlo errors with size $\simeq M^{-1/2}=10^{-4}$, which confirms that the PSK-ROCK method has no bias at equilibrium for the invariant measure in the absence of damping, as shown in \eqref{eq:exact}.
	We emphasise however that this exactness results holds only for linear problems, and a positive damping parameter $\eta$ should be used for nonlinear SDEs for stabilization, as shown in Sections~\ref{sec:analysis} and~\ref{sub:exactOU}.

\subsection{Nonglobally Lipschitz Brownian dynamics}
\label{sec:nonlip}
\begin{figure}[tb]
		\centering
				\begin{subfigure}[t]{0.45\textwidth}
			\includegraphics[width=1\linewidth]{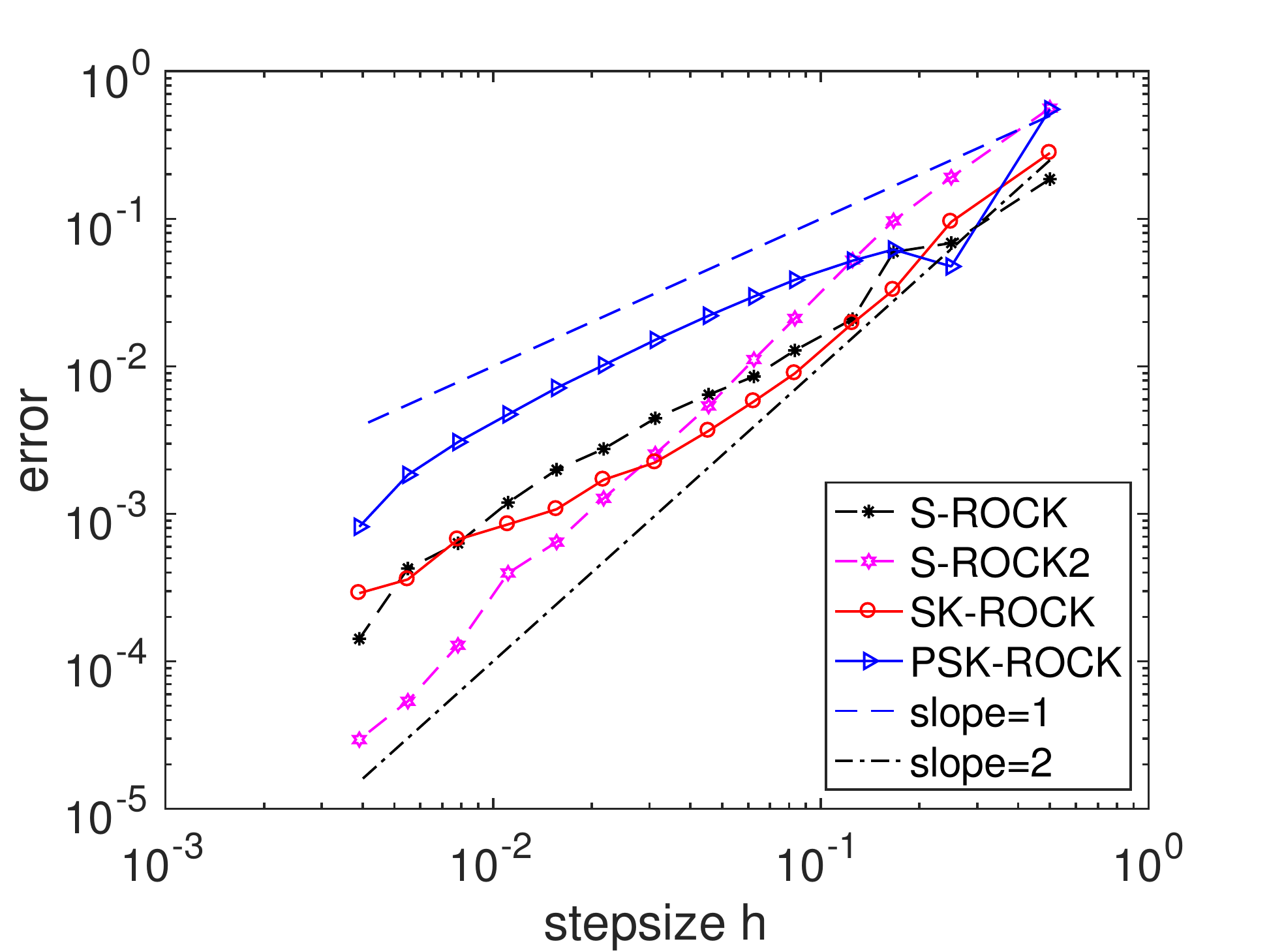}
			\caption{Final time $T=0.5$, $h=T/[2^{i/2}],\\i=1,\ldots,14.$}
			\label{fig:snoushorth}
		\end{subfigure}
				\begin{subfigure}[t]{0.45\textwidth}
			\includegraphics[width=1\linewidth]{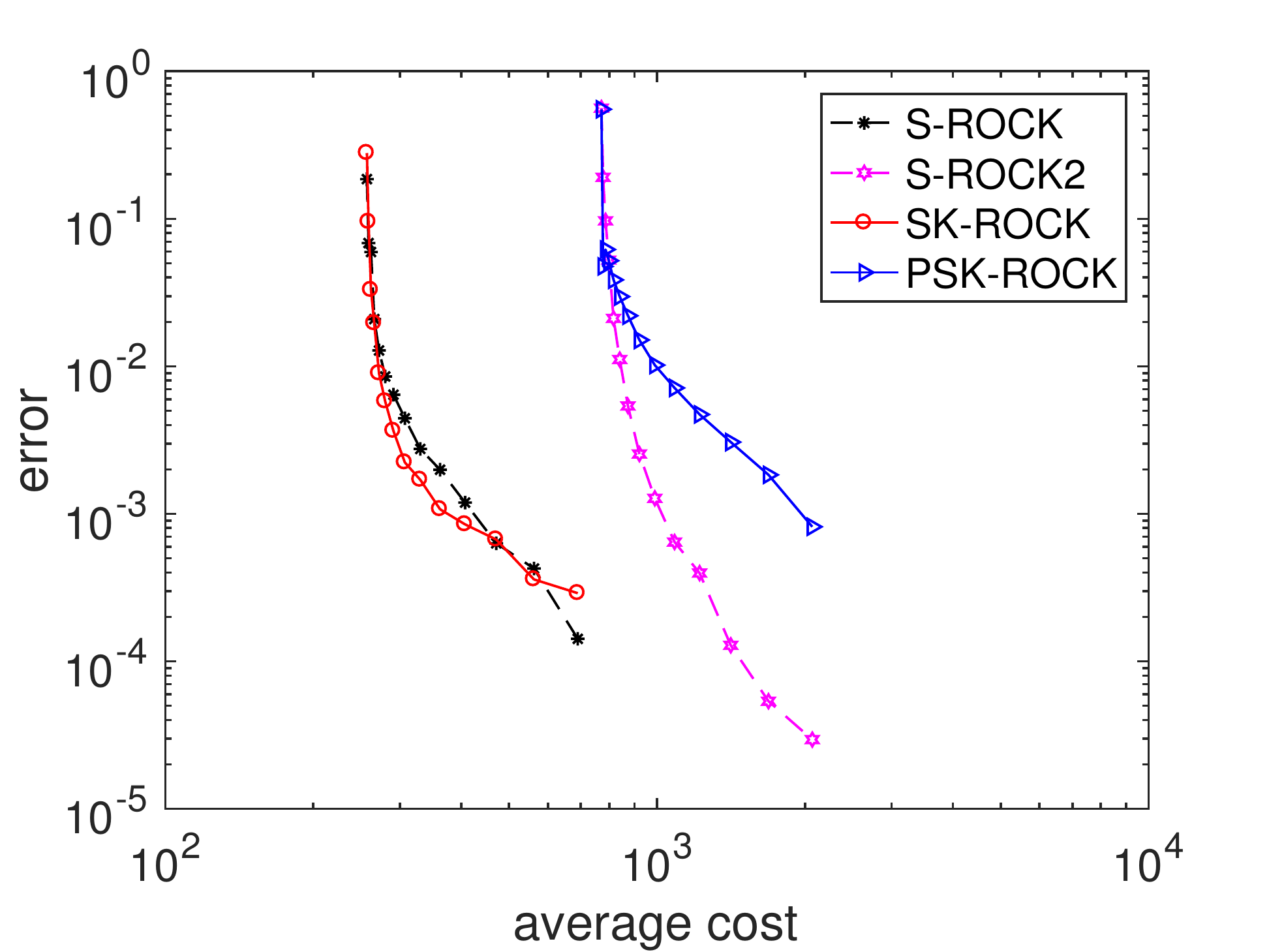}
			\caption{Final time $T=0.5$, $h=T/[2^{i/2}],\\i=1,\ldots,14.$}
			\label{fig:snoushort}
		\end{subfigure}\\[-3ex]
		\begin{subfigure}[t]{0.45\textwidth}
        \includegraphics[width=1\linewidth]{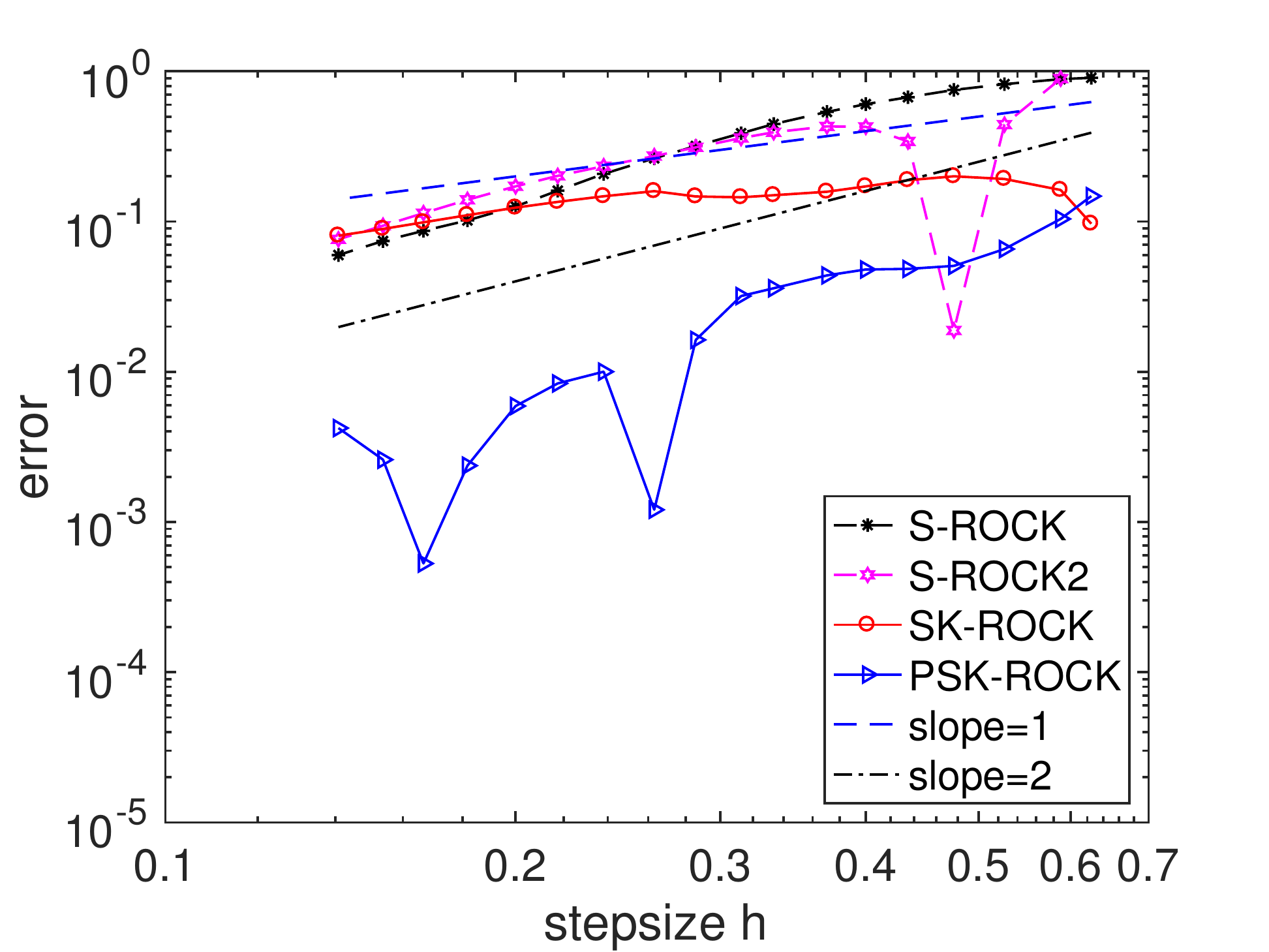}
		\caption{Final time $T=10$, $h=T/[15\times2^{i/8}],\\i=1,\ldots,18.$}
		\label{fig:snoulongh}
		\end{subfigure}
		\begin{subfigure}[t]{0.45\textwidth}
        \includegraphics[width=1\linewidth]{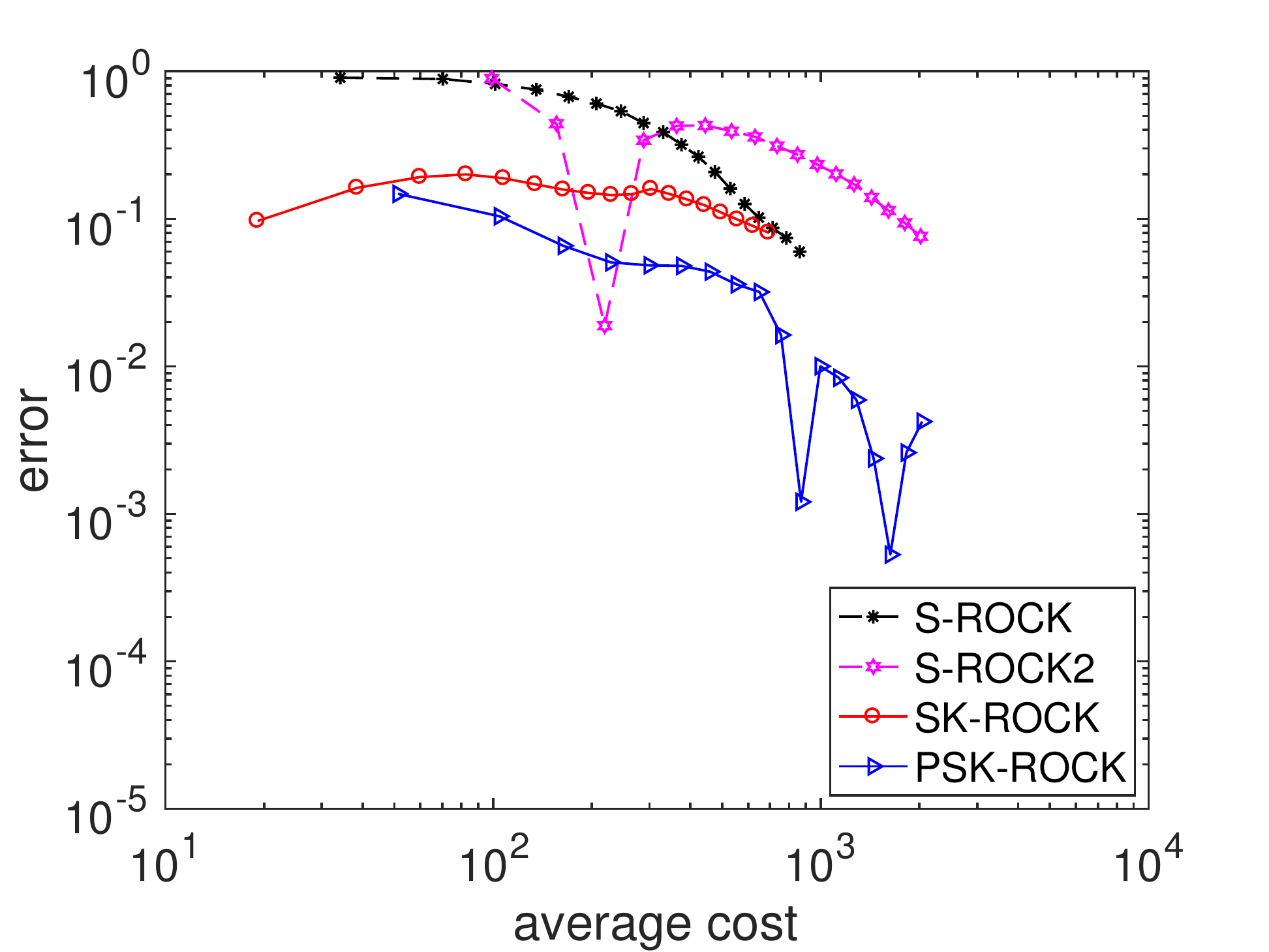}
		\caption{Final time $T=10$, $h=T/[15\times2^{i/8}],\\i=1,\ldots,18.$}
		\label{fig:snoulong}
		\end{subfigure}
	    
	\caption{Second moment errors versus the average number of drift function evaluations for problem \eqref{eq:snou} using S-ROCK, S-ROCK2 and the new method SK-ROCK and its postprocessed version PSK-ROCK. We use  discrete random increments and $10^8$ samples. }
	\label{fig:snoufig}
	\end{figure}
To illustrate the advantage of the PSK-ROCK method applied to nonglobally Lipschitz ergodic Brownian dynamics, we next consider the following double well potential $V(x)=(1-x^2)^2/4$ and the corresponding one-dimensional Brownian dynamics problem
	\begin{equation}\label{eq:snou}
	dX=(-X^3+X)dt+\sqrt2dW,\qquad X(0)=0,
	\end{equation}
	In Figure \ref{fig:snoufig}, we compare the performances of S-ROCK, 
	S-ROCK2 considered in \cite{AVZ13} (a method with weak order 2 for general SDEs), and the new SK-ROCK and PSK-ROCK methods at short time $T=0.5$ (Figures \ref{fig:snoufig}(a)(b)) and long time $T=10$ (Figures \ref{fig:snoufig}(c)(d)). 
	As we focus on invariant measure convergence and not on strong convergence, we consider here discrete random increments
	with $\IP(\xi_n=\pm \sqrt3)=1/6,\IP(\xi_n=0)=2/3$, which has the correct moments so that 
	Theorem \ref{thm:postprocessor} remains valid. Our numerical tests indicate that it makes
	PSK-ROCK with modified stage \eqref{eq:defR1new} more stable. For a fair comparison, we use
	the same discrete random increments for all schemes.
	We plot the second moment error versus the time stepsize $h$ and versus the average cost which is the total number of function evaluations during the time integration divided by the total number number of samples. Indeed, the number of function evaluations depends on the trajectories because the stage parameter s is adaptive at each time step. 
	For short time, we can see that the S-ROCK and the SK-ROCK method have order 1 (Figure \ref{fig:snoufig}(a)) and exhibit similar performance with nearly identical error versus cost curves in Figure \ref{fig:snoufig}(b), while PSK-ROCK 
is less advantageous for short time. This illustrates that the postprocessing has no advantage for short times. The S-ROCK2 method is the most accurate for small time steps, and it has order 2 as shown in Figures \ref{fig:snoufig}(a)(c), but at the same time it has a larger average cost as observed in Figures \ref{fig:snoufig}(b)(d) due to its smaller stability domain with size $\simeq 0.42 \cdot s^2$.
For long time, the SK-ROCK and S-ROCK both exhibit order 1 of accuracy (Figure \ref{fig:snoufig}(c)), with an advantage in terms of error versus cost for the
SK-ROCK method that is about $10$ times more accurate for large time steps. In contrast, the postprocessed scheme PSK-ROCK exhibits order 2 of convergence (Figure \ref{fig:snoufig}(c)) which corroborates Theorem \ref{thm:postprocessor}. Since the postprocessing overcost is negligible (two additional vector field evaluations per timestep due to the modified stage $K_1$ in \eqref{eq:defR1new}), this makes PSK-ROCK the most efficient in terms of error versus cost, as shown in Figure \ref{fig:snoufig}(d). The S-ROCK2 method has order 2 here but with poor accuracy compared to the PSK-ROCK method with approximately the same cost. Note that typically the SK-ROCK method used $s=1,2,3$ stages in contrast to the S-ROCK method using $s=2,\ldots,6$ stages per timesteps.

	\subsection{Stochastic heat equation with multiplicative space-time noise}	
	Although our analysis applies only to finite dimensional systems of SDEs, we consider the following stochastic partial differential equation (SPDE) obtained by adding multiplicative noise to the heat equation,
	\begin{align}\label{eq:spde}
	\dif{u(t,x)}{t}&=\diff{u(t,x)}{x} + u(t,x)\dot{W}(t,x), \quad
	(t,x)\in[0,T]\times[0,1]\nonumber\\
	u(0,x)&=5\cos(\pi x),\quad x\in[0,1],\nonumber\\
	u(t,0)&=5,\quad \dif{u(t,1)}{x}=0,\quad t\in[0,T],
	\end{align}
		where $\dot{W}(t,x)$ denotes a space-time white noise that we discretize together with the Laplace operator with a standard finite difference formula \cite{DaG00}.
	We obtain the following stiff system of SDEs where
$u(x_i, t) \approx u_i(t)$, with $x_i=i \Delta x$, $\Delta x = 1/N$,
\begin{eqnarray*} 
du_i &=& \frac{u_{i+1}-2u_i+u_{i-1}}{\Delta x^2}dt+ \frac{u_i}{\sqrt{\Delta x}}  dw_i, \quad i=1,\ldots,N,
\end{eqnarray*}
where the Dirichlet and the Neumann conditions impose $u_{0}=5$ and $u_{N+1}=u_{N-1}$, respectively. 
Here, $w_1,\ldots, w_{N}$ are independent standard Wiener processes and $dw_i$ indicates It\^o noise. 
	In Figure \ref{fig:spde}(a), we plot one realization of the SPDE using space stepsize $\Delta x=1/100$ and timestep size $\Delta t=1/50$. Note that the Lipchitz constant associated to the space-discretization of \eqref{eq:spde} has size $\rho=4\Delta x^{-2}$, and the stability condition is fulfilled for $s=22$ stages. For comparison, the standard S-ROCK method would require $s=46$ stages, while applying the standard Euler-Maruyama with a smaller stable timestep $\Delta t/s$ would require $s \geq \Delta t \rho/2=400$ intermediate steps.
	Notice that the initial condition in \eqref{eq:spde} satisfies the boundary conditions, which permits a smooth solution close to time $t=0$. Taking alternatively an initial condition that does not satisfy the boundary conditions (for instance $u(x,0)=1$) yields an inaccurate numerical solution with large oscillations close to the boundary $x=0$. A simple remedy in such a case is to consider a larger damping parameter $\eta$, as described in Remark \ref{rem:stabilization}.
	
	
		\begin{figure}[tb]
		\centering
		\begin{subfigure}[t]{0.45\textwidth}
			\centering
			\includegraphics[scale=0.5]{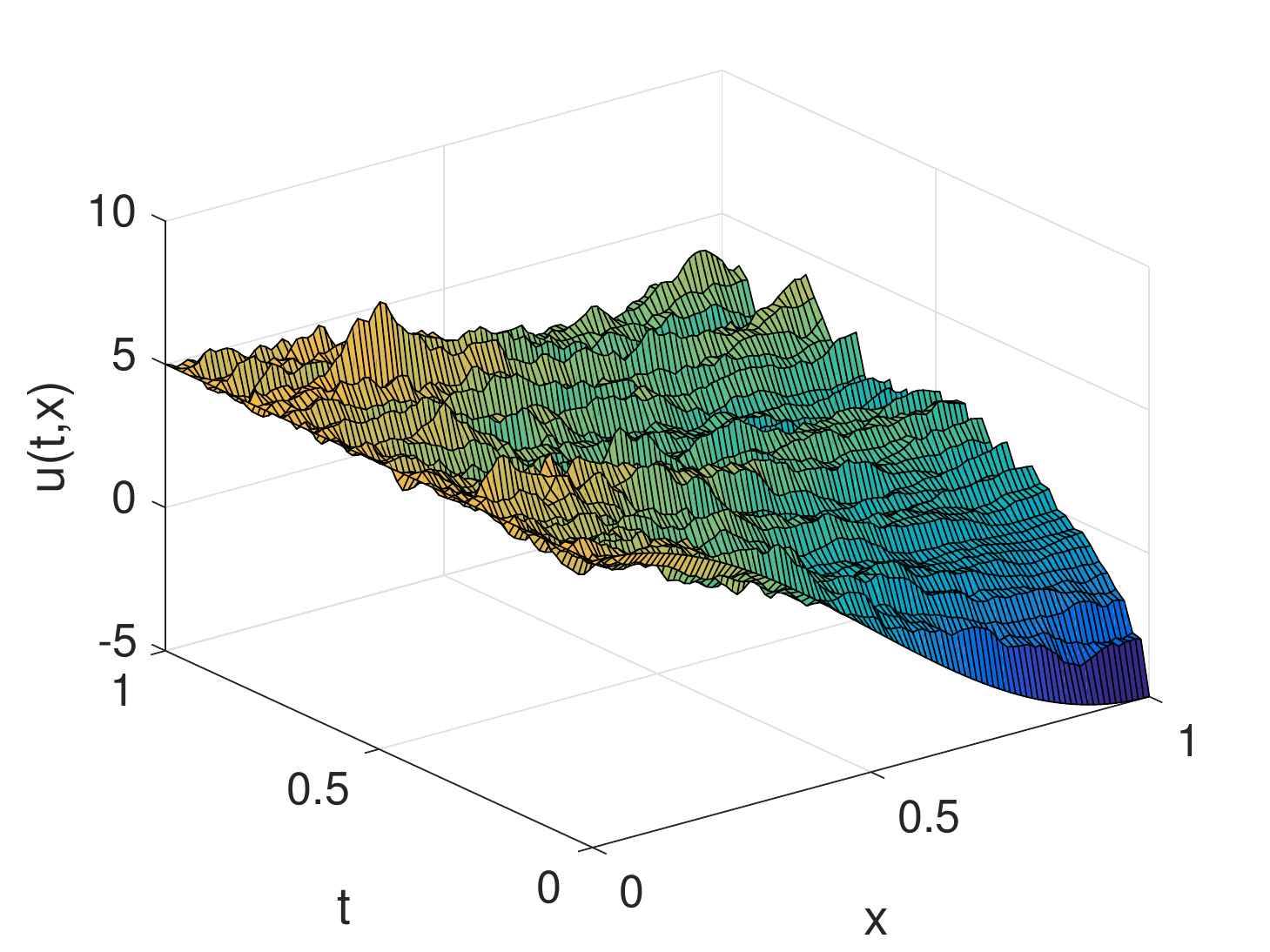}
			\caption{One realization with SK-ROCK using ${\Delta t=1/50}$, $\Delta x=1/100$, $s=22$.}
			\label{fig:solspde}
		\end{subfigure}
		\hfill
		\begin{subfigure}[t]{0.45\textwidth}
			\centering
			\includegraphics[width=1\linewidth]{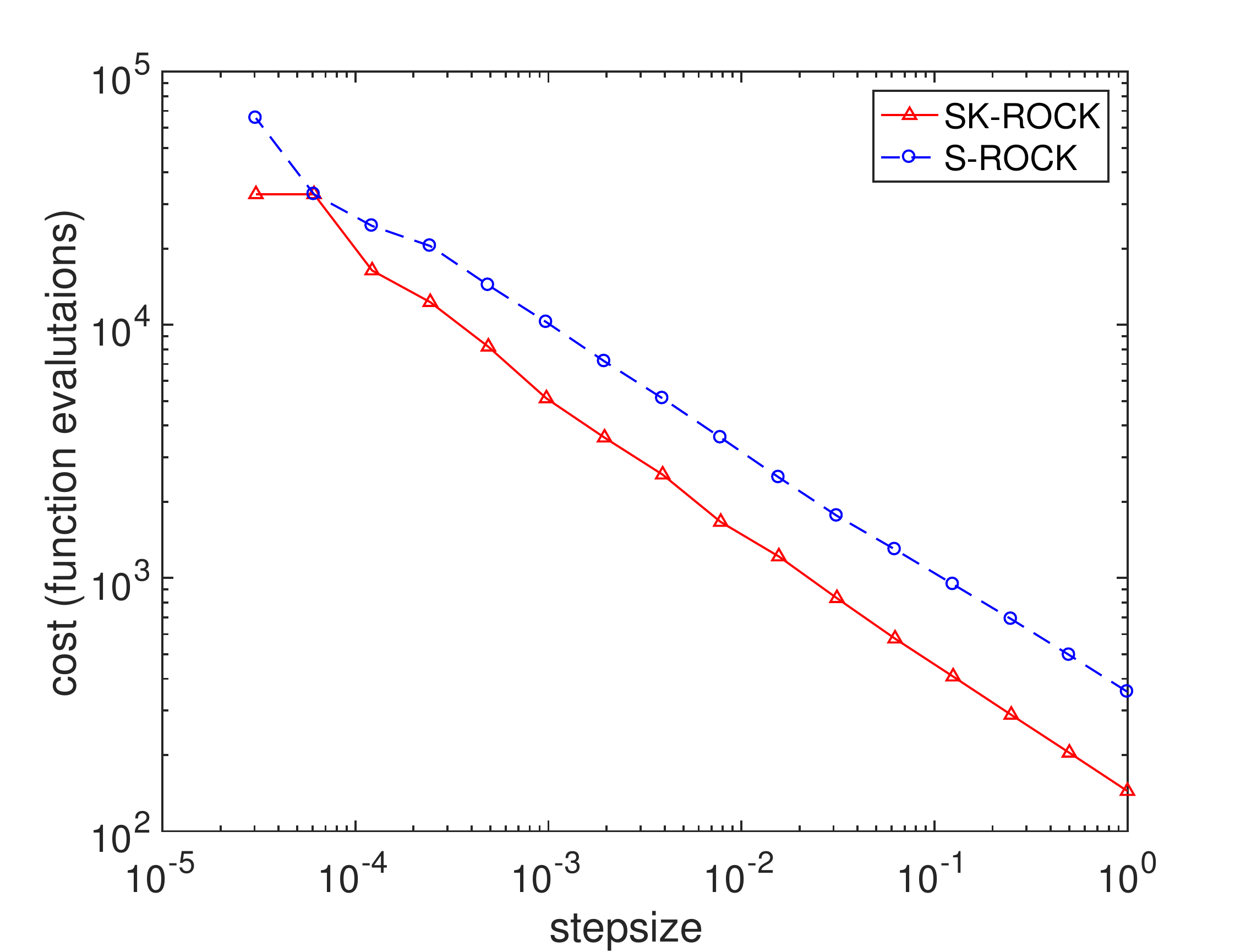}
			\caption{Comparison of S-ROCK and SK-ROCK: cost (function evaluations) with respect to stepsize $\Delta t=2^{-i}, i=0,\ldots,15$.}
			\label{fig:costb}
		\end{subfigure}
		\caption{
		SPDE problem \eqref{eq:spde} using the space discretization stepsize $\Delta x=1/100$.
			}
		\label{fig:spde}
	\end{figure}
	
	In Figure \ref{fig:spde}(b), we compare the number of vector field evaluations of the standard S-ROCK and new SK-ROCK methods when applied to the SPDE \eqref{eq:spde} with finite difference discretization with parameter $\Delta x=1/100$. The better performance of SK-ROCK with damping $\eta=0.05$ is due to its larger stability domain with size $\simeq 1.94\cdot s^2$ compared to the size $\simeq 0.33 \cdot s^2$ for S-ROCK. Observing the ratio of the two costs in Figure \ref{fig:spde}(b), we see that the new SK-ROCK methods has a reduced cost for stabilization by an asymptotic factor of about $\sqrt{1.94/0.33}\simeq 2.4$ for large $s$ and large stepsizes, which confirms the stability analysis of Section \ref{sec:analysis}.
	The convergence analysis of the SK-ROCK method for the stochastic heat equation is the topic of future work.
	\begin{remark}
Notice that SK-ROCK with $s=1$ stage has the optimal mean-square stability length ($L=2$ for $\eta=0$) as defined in \eqref{eq:defL}. In contrast, the S-ROCK method with $s=1$ has the smaller stability length $L=3/2$, while the standard Euler-Maruyama has $L=0$.
This explains why for the smallest considered stepsize $\Delta t=2^{-15}$ in Figure \ref{fig:spde}(b), we have $s=1$ for SK-ROCK while 
S-ROCK uses $s=2$ stages.
	\end{remark}
	\begin{figure}[tb]
		\centering
		\begin{subfigure}[t]{0.45\textwidth}
			\centering
			\includegraphics[scale=0.5]{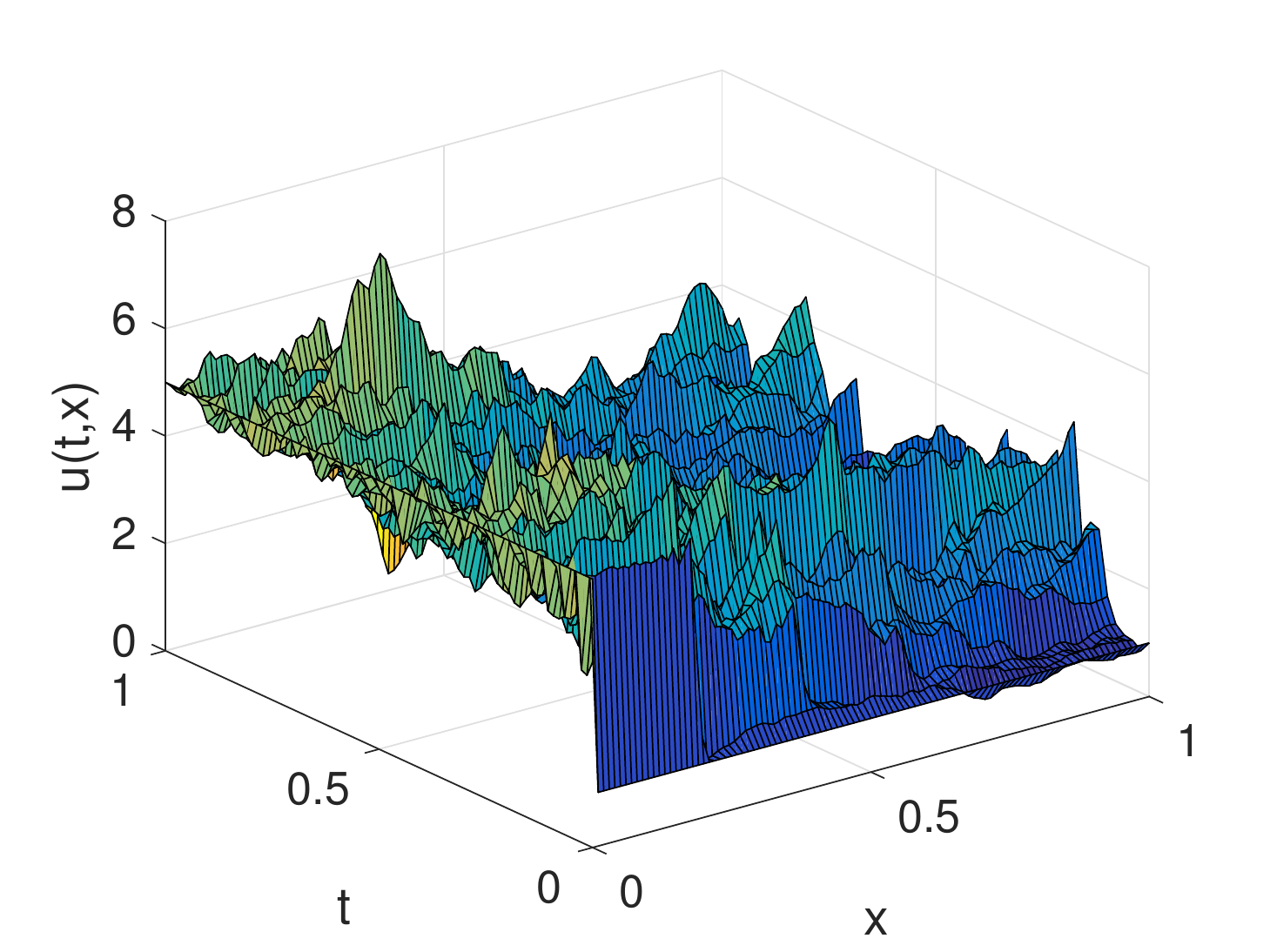}
			\caption{$\eta=0.05$, $s=22$.}
			\label{fig:spdestiffa}
		\end{subfigure}
		\qquad
		\begin{subfigure}[t]{0.45\textwidth}
			\centering
			\includegraphics[scale=0.5]{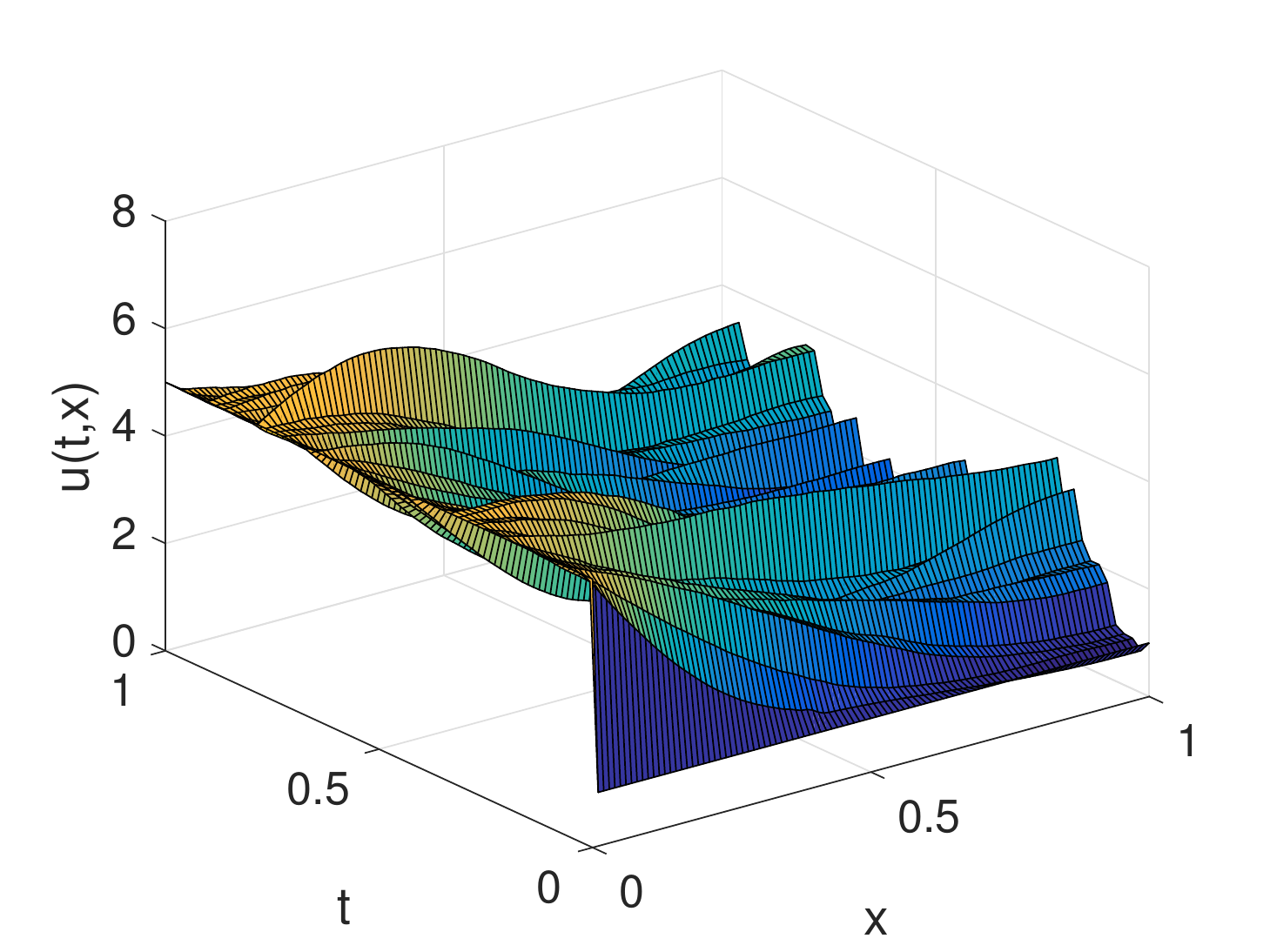}
			\caption{$\eta=10$, $s=44$.}
			\label{fig:spdestiffb}
		\end{subfigure}
		
		\caption{
		SPDE problem \eqref{eq:spde} with the initial condition $u(0,x)=1$. One realization with SK-ROCK using ${\Delta t=1/50}$, $\Delta x=1/100$ for different values of the damping parameter $\eta$.
			}
		\label{fig:spdestiff}
	\end{figure}
	In Figure \ref{fig:spdestiff} we consider again one realization with SK-ROCK of the SPDE problem \eqref{eq:spde} but with a different initial condition
	$u(0,x)=1$ not fulfilling the boundary conditions, i.e. that is outside the domain of the Laplace operator, as considered in \cite{AbB2}. 
	We compare the result for the same sets of random numbers but for 
	for different values of the damping parameter $\eta$. We observe numerically high oscillations in time and space for the small damping value $\eta=0.05$  in Figure \eqref{fig:spdestiffa} while the larger damping $\eta=10$ yields a smoother solution in Figure \eqref{fig:spdestiffb}.
	This illustrates again Remark~\ref{rem:stabilization} showing that the damping parameter $\eta$ can be increased in the case of
	severely stiff problems, adjusting the stage parameter accordingly with \eqref{eq:defseta}.
	
\bigskip

\noindent \textbf{Acknowledgements.}\
The work of the first author was partially supported by the Swiss National Foundation, grant
200020\_172710. The work of the second and third authors was partially supported by the Swiss National Science Foundation, grants 200020\_144313/1 and 200021\_162404. 
The computations were performed at University of Geneva on the Baobab cluster.

\bibliographystyle{abbrv}
	\bibliography{abd_biblio,HLW,complete}

\def\cprime{$'$} \def\cprime{$'$} \def\cprime{$'$}
\begin{thebibliography}{10}

\bibitem{Abd02}
A.~Abdulle.
\newblock Fourth order {C}hebyshev methods with recurrence relation.
\newblock {\em SIAM J. Sci. Comput.}, 23(6):2041--2054, 2002.

\bibitem{Abd13c}
A.~Abdulle.
\newblock {\em Explicit Stabilized Runge--Kutta Methods}, pages 460--468.
\newblock Encyclopedia of Applied and Computational Mathematics, Springer
  Berlin Heidelberg, 2015.

\bibitem{AbB2}
A.~Abdulle and A.~Blumenthal.
\newblock Stabilized multilevel {M}onte {C}arlo method for stiff stochastic
  differential equations.
\newblock {\em J. Comput. Phys.}, 251:445--460, 2013.

\bibitem{AbC08}
A.~Abdulle and S.~Cirilli.
\newblock S-{ROCK}: {C}hebyshev methods for stiff stochastic differential
  equations.
\newblock {\em SIAM J. Sci. Comput.}, 30(2):997--1014, 2008.

\bibitem{AbL08}
A.~Abdulle and T.~Li.
\newblock {S-ROCK} methods for stiff {I}to {SDEs}.
\newblock {\em Commun. Math. Sci.}, 6(4):845--868, 2008.

\bibitem{AbM01}
A.~Abdulle and A.~Medovikov.
\newblock Second order chebyshev methods based on orthogonal polynomials.
\newblock {\em Numer. Math.}, 90(1):1--18, 2001.

\bibitem{AVZ12b}
A.~Abdulle, G.~Vilmart, and K.~C. Zygalakis.
\newblock Mean-square {$A$}-stable diagonally drift-implicit integrators of
  weak second order for stiff {I}t\^o stochastic differential equations.
\newblock {\em BIT}, 53(4):827--840, 2013.

\bibitem{AVZ13}
A.~Abdulle, G.~Vilmart, and K.~C. Zygalakis.
\newblock Weak second order explicit stabilized methods for stiff stochastic
  differential equations.
\newblock {\em SIAM J. Sci. Comput.}, 35(4):A1792--A1814, 2013.

\bibitem{BrV16}
C.-E. Br{\'e}hier and G.~Vilmart.
\newblock High {O}rder {I}ntegrator for {S}ampling the {I}nvariant
  {D}istribution of a {C}lass of {P}arabolic {S}tochastic {PDE}s with
  {A}dditive {S}pace-{T}ime {N}oise.
\newblock {\em SIAM J. Sci. Comput.}, 38(4):A2283--A2306, 2016.

\bibitem{BuC10}
E.~Buckwar and C.~Kelly.
\newblock Towards a systematic linear stability analysis of numerical methods
  for systems of stochastic differential equations.
\newblock {\em SIAM Journal on Numerical Analysis}, 48(1):298--321, 2010.

\bibitem{BBT04}
K.~Burrage, P.~Burrage, and T.~Tian.
\newblock Numerical methods for strong solutions of stochastic differential
  equations: an overview.
\newblock {\em Proc. R. Soc. Lond. Ser. A Math. Phys. Eng. Sci.},
  460(2041):373--402, 2004.

\bibitem{butcher69teo}
J.~C. Butcher.
\newblock The effective order of {R}unge-{K}utta methods.
\newblock In J.~L. Morris, editor, {\em Proceedings of Conference on the
  Numerical Solution of Differential Equations}, volume 109 of {\em Lecture
  Notes in Math.}, pages 133--139, 1969.

\bibitem{ChW12}
Y.~Chong and J.~B. Walsh.
\newblock The roughness and smoothness of numerical solutions to the stochastic
  heat equation.
\newblock {\em Potential Anal.}, 37(4):303--332, 2012.

\bibitem{DaG00}
A.~M. Davie and J.~G. Gaines.
\newblock Convergence of numerical schemes for the solution of parabolic
  stochastic partial differential equations.
\newblock {\em Math. Comp.}, 70(233):121--134, 2001.

\bibitem{Gar88}
T.~Gard.
\newblock {\em Introduction to stochastic differential equations}.
\newblock Marcel Dekker, New York, 1988.

\bibitem{HaW96}
E.~Hairer and G.~Wanner.
\newblock {\em Solving ordinary differential equations II. Stiff and
  differential-algebraic problems}.
\newblock Springer-Verlag, Berlin and Heidelberg, 1996.

\bibitem{Hig00}
D.~J. Higham.
\newblock Mean-square and asymptotic stability of the stochastic theta method.
\newblock {\em SIAM J. Numer. Anal.}, 38(3):753--769, 2000.

\bibitem{KlP92}
P.~Kloeden and E.~Platen.
\newblock {\em Numerical solution of stochastic differential equations}.
\newblock Springer-Verlag, Berlin and New York, 1992.

\bibitem{LaV17}
A.~Laurent and G.~Vilmart.
\newblock Exotic aromatic {B}-series for the study of long time integrators for
  a class of ergodic {SDEs}.
\newblock {\em Submitted}, 2017.
\newblock arXiv:1707.02877.

\bibitem{LM13}
B.~Leimkuhler and C.~Matthews.
\newblock Rational construction of stochastic numerical methods for molecular
  sampling.
\newblock {\em Appl. Math. Res. Express.}, 2013(1):34--56, 2013.

\bibitem{LMT14}
B.~Leimkuhler, C.~Matthews, and M.~V. Tretyakov.
\newblock On the long-time integration of stochastic gradient systems.
\newblock {\em Proc. R. Soc. A}, 470(2170), 2014.

\bibitem{Mil87}
G.~N. Mil{\cprime}shte{\u\i}n.
\newblock A theorem on the order of convergence of mean-square approximations
  of solutions of systems of stochastic differential equations.
\newblock {\em Teor. Veroyatnost. i Primenen.}, 32(4):809--811, 1987.

\bibitem{Mil86}
G.~Milstein.
\newblock Weak approximation of solutions of systems of stochastic differential
  equations.
\newblock {\em Theory Probab. Appl.}, 30(4):750--766, 1986.

\bibitem{MiT04}
G.~Milstein and M.~Tretyakov.
\newblock {\em Stochastic numerics for mathematical physics}.
\newblock Scientific Computing. Springer-Verlag, Berlin and New York, 2004.

\bibitem{SaM96}
Y.~Saito and T.~Mitsui.
\newblock Stability analysis of numerical schemes for stochastic differential
  equations.
\newblock {\em SIAM J. Numer. Anal.}, 33:2254--2267, 1996.

\bibitem{Toc05}
A.~Tocino.
\newblock Mean-square stability of second-order {R}unge-{K}utta methods for
  stochastic differential equations.
\newblock {\em J. Comput. Appl. Math.}, 175(2):355--367, 2005.

\bibitem{HoS80}
P.~{Van der Houwen} and B.~Sommeijer.
\newblock On the internal stage {R}unge-{K}utta methods for large m-values.
\newblock {\em Z. Angew. Math. Mech.}, 60:479--485, 1980.

\bibitem{Ver96b}
J.~Verwer.
\newblock Explicit {R}unge-{K}utta methods for parabolic partial differential
  equations.
\newblock {\em Special issue of Appl. Num. Math.}, 22:359--379, 1996.

\bibitem{Vil15}
G.~Vilmart.
\newblock Postprocessed integrators for the high order integration of ergodic
  {SDE}s.
\newblock {\em SIAM J. Sci. Comput.}, 37(1):A201--A220, 2015.

\end{thebibliography}

\end{document}